\documentclass[a4paper]{amsart}                % American Mathematical Society document class for articles
\usepackage[latin1]{inputenc}                  % Codificacion europea del teclado (acentos, simbolos teclado ...)
\usepackage[T1]{fontenc}                       % Acentos de palabras y division silabica
\usepackage[spanish,english]{babel}            % Idioma principal el ingles
\usepackage[pdftex]{graphicx}                  % Incorporar graficos
\usepackage{amsmath,amssymb,amsthm}            % Paquetes de la American Mathematical Society
\usepackage{latexsym}                          % Simbolos
\usepackage{delarray}                          % Delimiter array: permite integrar delimitadores en las opciones del entorno array
\usepackage{bbm}                               % Fuentes para los numeros reales, racionales, ...
\usepackage{hyperref}                          % Enlaces de "hipertexto"
\usepackage[pdftex,usenames,dvipsnames]{color} % Muy importante para el color
\usepackage{bbding,trfsigns}                   % Simbolos
\usepackage{wasysym}                           % Variar el tamaño de los simbolos
\usepackage{datetime}                          % Imprimir fecha y hora
\usepackage{multirow}
\usepackage{rotating}
\usepackage{booktabs}

\DeclareMathAlphabet{\mathpzc}{OT1}{pzc}{m}{it} % Tipo de letra \mathpzc
\newcommand{\B}{\mathbb{B}}

\newcommand{\D}{\mathcal{D}}
\newcommand{\DD}{\mathbb{D}}

\newcommand{\M}{\mathcal{M}}
\newcommand{\N}{\mathbb{N}}
\newcommand{\R}{\mathbb{R}}

\newtheorem{ThA}{Theorem}

\newtheorem{Th}{Theorem}[section]              % Enumera los teoremas de acuerdo con la seccion (Theorem 1.1, Theorem 1.2 , ...)

\newtheorem{Rem}[Th]{Remark}

\newtheorem{Lem}[Th]{Lemma}

\title[UMD Banach spaces and the maximal regularity]
      {UMD Banach spaces and the maximal regularity for the square root of several operators}
\author[V. Almeida]{Víctor Almeida}
\author[J.J. Betancor]{Jorge J. Betancor}
\author[A.J. Castro]{Alejandro J. Castro}
\address{\newline
        Víctor Almeida, Jorge J. Betancor, Alejandro J. Castro \newline
        Departamento de An\'alisis Matem\'atico,
        Universidad de La Laguna, \newline
        Campus de Anchieta, Avda. Astrof\'{\i}sico Francisco S\'anchez, s/n, \newline
        38271, La Laguna (Sta. Cruz de Tenerife), Spain}
\email{valmeida@ull.es, jbetanco@ull.es, ajcastro@ull.es}


\thanks{The second and the third authors are partially supported by MTM2010/17974.
The third author is also supported by a FPU grant from the Government of Spain.}

\begin{document}

\footnotetext{Date: \today.}

\maketitle                                  % Si no se activa esta opción no se pone ni el titulo ni los autores en el encabezado de cada página

\begin{abstract}
    In this paper we prove that the maximal $L^p$-regularity property on the interval $(0,T)$, $T>0$,
    for Cauchy pro\-blems associated with the square root of Hermite, Bessel or Laguerre type operators on $L^2(\Omega , d\mu; X),$
    characterizes the UMD property for the Banach space $X$.
\end{abstract}

%%%%%%%%%%%%%%%%%%%%%%%%%%%%%%%%%%%%%%%%%%%%%%%%%%%%%%%%%%%%%%%%%%%%%%%%%%%%%%%%%%%%%%%%%%%%%%%%%%%%%%%%%%%%%%%%%%%%
\section{Introduction}\label{sec:intro}
%%%%%%%%%%%%%%%%%%%%%%%%%%%%%%%%%%%%%%%%%%%%%%%%%%%%%%%%%%%%%%%%%%%%%%%%%%%%%%%%%%%%%%%%%%%%%%%%%%%%%%%%%%%%%%%%%%%%

Suppose that $\B$ is a complex Banach space, $A$ is a (possible unbounded) operator on $\B$ such that $D(A)$ is dense in $\B$.
We consider the Cauchy problem
\begin{equation}\label{eq INTRO1}
    \left\{\begin{array}{ll}
    u'(t)+Au(t)=f(t), & 0 \leq t < T, \\
    u(0)=0, &
    \end{array}\right.
\end{equation}
where $f:(0,T)\rightarrow \B$ is measurable and $0 < T \leq \infty$.

Let $1<p<\infty$ and $0 < T \leq \infty$. We say that the operator $A$ has the maximal $L^p$-regularity property on the interval $(0,T)$
when for every $f\in L^p((0,T),\B)$ there exists a unique solution $u\in L^p((0,T),\B)$ of \eqref{eq INTRO1}, and
\begin{equation*}\label{eq INTRO2}
\|u'\|_{L^p((0,T),\B)}+\|Au\|_{L^p((0,T),\B)}\leq C\|f\|_{L^p((0,T),\B)},
\end{equation*}
for a certain $C>0$ which does not depend on $f$. This property of maximal regularity is important, for instance, in nonlinear problems
and a lot of authors have studied it in the last years (see \cite{Am}, \cite{AB1}, \cite{AB2}, \cite{AB3}, \cite{AA}, \cite{AMP}, \cite{KL}, \cite{Le} and \cite{W}, amongst others).

Assume that $-A$ generates a $C_0$-semigroup $\{T_t\}_{t>0}$ on $\B$. Then the unique mild solution of \eqref{eq INTRO1} is defined by
\begin{equation*}\label{eq INTRO3}
    u(t)=\int_0^t{T_{t-s}(f(s))ds}, \quad t\in (0,T),
\end{equation*}
for every $f\in L^p((0,T),\B)$.

In order to prove the maximal $L^p$-regularity for the operator $A$, we consider the operator $K(\{T_u\}_{u>0})$ defined by
\begin{equation}\label{D11}
    K(\{T_u\}_{u>0})f(t)=\int_0^t{\partial_tT_{t-s}(f(s))ds}, \quad t>0.
\end{equation}
Since $\|\partial_t T_t\|_{L(\B)} \leq C/t$, $t>0$, where $L(\B)$ denotes the space of bounded operators from $\B$ into itself, the
convergence of the integral in \eqref{D11} depends on the properties of $f$. Suppose that $f \in C^1_c((0,T),D(A))$, that is,
$f:(0,T) \longrightarrow D(A)$ is a $C^1$-function with compact support. Then, we have that
\begin{align*}
    \int_0^t \left\| \partial_t T_{t-s}(f(s)) \right\|_\B ds
        \leq & C \left( \int_0^t \frac{\left\| f(s)-f(t) \right\|_\B}{t-s} ds + t \|Af(t)\|_\B \right) \\
        \leq & C t \left( \|f'\|_{L^\infty((0,T),\B)} + \|Af(t)\|_\B \right), \quad t>0.
\end{align*}
The subspace $C_c^1(0,T)\otimes D(A)$ is dense in $L^p((0,T),\B)$, $1<p<\infty$, because $D(A)$ is dense in $\B$.
The operator $A$ has the maximal $L^p$-regularity property if, and only if, the operator
$K(\{T_u\}_{u>0})$ can be extended from $C_c^1(0,T)\otimes D(A)$ to  $L^q((0,T),\B)$ as a bounded operator from $L^q((0,T),\B)$
into itself, for some (equivalently, for every) $1<q<\infty$. For the sake of simplicity, when the operator $K(\{T_u\}_{u>0})$ satisfies this
property, we say that it is bounded from $L^q((0,T),\B)$ into itself.

The Banach spaces with the UMD property play an important role in the study of maximal $L^p$-regularity for operators.
It is well known that the Hilbert transform $\mathcal{H}$ defined, for every $f\in L^p(\R)$, $1<p<\infty$,
$$\mathcal{H}(f)(x)
    =\lim_{\varepsilon \to 0^+} \frac{1}{\pi} \int_{|x-y|>\varepsilon}{{{f(y)}\over{x-y}}dy}, \quad a.e. \ x\in\R,$$
is bounded from $L^p(\R)$ into itself, $1<p<\infty$. If $X$ is a Banach space, $\mathcal{H}$ is extended to $L^p(\R)\otimes X$, $1<p<\infty$,
in a natural way. We say that $X$ is UMD when the Hilbert transform can be extended from $L^p(\R)\otimes X$ to $L^p(\R,X)$ as a bounded operator from
$L^p(\R,X)$ into itself for some (equivalently, for any) $1<p<\infty$. The main properties of UMD Banach spaces can be encountered in \cite{Bou}, \cite{Bu3}
and \cite{Rub}.

It is well-known that the operator $-\sqrt{\Delta}$ generates a semigroup, so called classical Poisson semigroup, being $\Delta=-d^2/dx^2$.
For every $t>0$, we represent, by $P_t(f)$ the classical Poisson integral of
$f\in L^p(\R),$ $1\leq p < \infty$, that is,
$$P_t(f)(x)=\int_{\R}{P_t(x-y)f(y)dy}, \quad x\in\R,$$
where
$$P_t(x)={1\over{\pi}}{t\over{t^2+x^2}}, \quad  x\in\R.$$
As it is also well-known, if $S$ is a positive bounded operator from $L^p(\Omega , d\mu)$ into itself, where $1\leq p<\infty$ and $(\Omega , d\mu)$
is a measure space, the operator $S\otimes I_X$ can be extended, for every Banach space $X$, to the Lebesgue-Bochner space
$L^p(\Omega , d\mu; X)$ as a bounded operator, that we continue denoting by $S$, from $L^p(\Omega , d\mu; X)$ into itself.
Thus, $P_t$ is a positive operator and $P_t \otimes I_X$ can be extended to $L^2(\R,X)$, for every $t>0$.

As usual, we adopt the same nomenclature when we are dealing with the square root of our operators, that is, we call Poisson semigroup to the
ones generated by the square root of them.

Brezis (see \cite{CL}) posed the following question: given a Banach space $\B$, is it true that every negative generator $-A$ of a bounded analytic
semigroup on $\B$ has maximal $L^p$ regularity? A partial answer was given by Coulhon and Lamberton \cite[p. 160]{CL} who established the following result.

\begin{ThA}\label{ThA}
    Let $X$ be a Banach space. Then, the operator $\sqrt{\Delta}$ acting on $\B=L^2(\R,X)$, has maximal $L^p$-regularity on the interval $(0,T)$, for any $T>0$,
    if, and only if, $X$ is UMD.
\end{ThA}

Our objective in this paper is to establish analogous results  to Theorem~\ref{ThA} when the Laplacian is replaced by
Bessel, Hermite or Laguerre type operators which are shown in the following chart, using the same notation as in \cite{StiZh},
after making the corresponding translation of the parameter $\alpha>0$.\\

\begin{center}
   \begin{tabular}{| c | c | c | c | }
     \cline{2-4} \multicolumn{1}{c|}{}
     & $\mathcal{L}$ & $\Omega$ & $d \mu$ \\
     \hline
     \multirow{1}{3mm}{\begin{sideways}\parbox{30mm}{\centering Bessel type \quad}\end{sideways}} & & & \\
     & $\displaystyle S_\alpha = - \frac{d^2}{dx^2} + \frac{\alpha(\alpha-1)}{x^2} $ & & $dx$\\
     & & $(0,\infty)$ & \\
     & $\displaystyle \Delta_\alpha = - \frac{d^2}{dx^2} - \frac{2 \alpha}{x} \frac{d}{dx}$ & & $x^{2\alpha}dx$ \\
     & & & \\
     \hline
     \multirow{1}{3mm}{\begin{sideways}\parbox{30mm}{\centering Hermite type \qquad }\end{sideways}}  & & & \\
     & $\displaystyle H = - \frac{1}{2} \left( \frac{d^2}{dx^2} - x^2 \right) $ & & $dx$ \\
     & & $\R$ & \\
     & $\displaystyle \mathcal{O} + \frac{1}{2}= - \frac{1}{2} \left( \frac{d^2}{dx^2} - 2 x \frac{d}{dx} \right) + \frac{1}{2} $ & & $\displaystyle \pi^{-1/2}e^{-x^2} dx$\\
     & & & \\
     \hline
     & & & \\
     & $\displaystyle L_\alpha^\varphi = - \frac{1}{2} \left( \frac{d^2}{dx^2} - x^2 - \frac{\alpha(\alpha-1)}{x^2}\right) $ & & $dx$\\
     & & & \\
     \multirow{1}{3mm}{\begin{sideways}\parbox{30mm}{\centering Laguerre type}\end{sideways}}  & $\displaystyle L_\alpha + \frac{\alpha + 1/2 }{2} = - \frac{1}{2}\left(x \frac{d^2}{dx^2} + (\alpha + 1/2 - x) \frac{d}{dx}\right) + \frac{\alpha + 1/2 }{2}$ & & $x^{\alpha-1/2}e^{-x}dx$\\
     & & & \\
     & $\displaystyle L_\alpha^\ell = - \frac{1}{2} \left( x \frac{d^2}{dx^2}+ (\alpha + 1/2) \frac{d}{dx}  - \frac{x}{4} \right) $ & $(0,\infty)$ & $x^{\alpha -1/2}dx$\\
     & & & \\
     & $\displaystyle L_\alpha^\psi = - \frac{1}{2} \left( \frac{d^2}{dx^2} - x^2 + \frac{2\alpha}{x} \frac{d}{dx} \right)$ & & $x^{2\alpha}dx$ \\
     & & & \\
     & $\displaystyle L_\alpha^\mathpzc{L} = - \frac{1}{2} \left( x \frac{d^2}{dx^2} + \frac{d}{dx}  - \frac{x}{4} - \frac{(\alpha - 1/2)^2}{4x} \right) $ & & $dx$ \\
     & & & \\
     \hline
   \end{tabular}
\end{center}

 \quad \\

The main result of this paper is the following.

\begin{Th}\label{Th1}
    Let $X$ be a Banach space and $\mathcal{L}$ any of the operators listed above. The following assertions are equivalent.
    \begin{enumerate}
        \item[$(a)$] $X$ is UMD.
        \item[$(b)$] The operator $\sqrt{\mathcal{L}}$ acting on $L^2(\Omega,d\mu;X)$ has maximal $L^p$-regularity on the interval $(0,T)$, for any $T>0$.
    \end{enumerate}
\end{Th}

\begin{Rem}
    If $\mathcal{L}$ has non zero eigenvalue, which is the case of the Hermite and Laguerre type operators, we can take $T=\infty$ in Theorem~\ref{Th1}, $(b)$ (see, for instance
    \cite[Theorem 5.3.6]{ABHN} and  \cite[Corollary~4.2, Theorem~5.1 and Theorem~5.2]{Dor}).
\end{Rem}

If $\{T_t\}_{t>0}$ denotes the Poisson semigroup associated with $\mathcal{L}$, being $\mathcal{L}=S_\alpha$, $H$ or $L_\alpha^\varphi$,
to prove Theorem \ref{Th1} we will show that the operator $K(\{T_u\}_{u>0})$
is bounded from $L^2((0,T), L^2(\Omega , d\mu; X))$ into itself if, and only if, the operator $K(\{\tau_u\}_{u>0})$ is bounded from
$L^2((0,T), L^2(\Omega,d\mu;X))$ into itself, where $\tau_u=P_u$ when $\mathcal{L}=S_\alpha$ or $H$ and
$\tau_u=P_u^H$ when $\mathcal{L}=L_\alpha^\varphi$. Actually, we split the operator $K(\{T_u\}_{u>0})$ in two parts that we call local and global parts
$$K(\{T_u\}_{u>0})=K^{loc}(\{T_u\}_{u>0})+K^{glob}(\{T_u\}_{u>0}).$$
The operator $K^{glob}(\{T_u\}_{u>0})$ is bounded in $L^2((0,T), L^2(\Omega , d\mu; X))$ for every Banach space $X$.
On the other hand, the operator  $K^{loc}(\{T_u\}_{u>0})$ is bounded in $L^2((0,T), L^2(\Omega , d\mu; X))$ if, and only if, $K(\{\tau_u\}_{u>0})$
is bounded in  $L^2((0,T), L^2(\Omega,d\mu;X))$. Then, according to Theorem~\ref{ThA}, the $L^2((0,T), L^2(\Omega , d\mu; X))$-boundedness of $K(\{T_u\}_{u>0})$
characterizes the UMD property for $X$.

After that, and in order to prove Theorem~\ref{Th1} for the others operators contained in our chart,
we proceed as in \cite[p. 3160]{StiZh}, by making use of the following transference lemma.

\begin{Lem}\label{LemTransfer}
    Consider the operators $Uf(t,x)=M(x)f(t,x)$ and $Wf(t,x)=f(t,h(x))$, $x \in \Omega \subset \R$,  $t>0$,
    being $M, h \in C^\infty(\Omega)$ such that $M$ is positive and $h$ is a one-to-one function.
    Then, Theorem~\ref{Th1} $(b)$ is true for $\mathcal{L}$ if, and only if, it is true for
    $\overline{\mathcal{L}} = (U \circ W)^{-1} \circ \mathcal{L} \circ (U \circ W).$
\end{Lem}

\begin{proof}
    It is enough to prove that
    $K(\{P_u^{\mathcal{L}}\}_{u>0})$ is bounded from $L^2((0,T)\times \Omega,dt \times d\mu;X)$ into itself if, and only if,
    $K(\{P_u^{\overline{\mathcal{L}}}\}_{u>0})$ is bounded from $L^2((0,T)\times \Omega,dt \times d\overline{\mu};X)$ into itself.
    Here $d\overline{\mu}$ represents the measure $d\overline{\mu} = M(h^{-1}(x))^2 |J_{h^{-1}}|d\mu$, and $|J_{h^{-1}}|$ is the Jacobian of the
    inverse mapping $h^{-1}$.

    Suppose that $K(\{P_u^{\mathcal{L}}\}_{u>0})$ is bounded in $L^2((0,T)\times \Omega,dt \times d\mu;X)$.
    We can write, via the spectral theorem,
    \begin{align*}
        &\left\langle (U \circ W) P_t^{\overline{\mathcal{L}}} f, (U \circ W) g \right\rangle_{L^2((0,T)\times \Omega,dt \times d\mu;X)}
            = \left\langle P_t^{\overline{\mathcal{L}}} f, g \right\rangle_{L^2((0,T)\times \Omega,dt \times d\overline{\mu};X)} \\
        & \qquad  = \int_0^\infty e^{-t \sqrt{\lambda}} d\overline{E}_{f,g}(\lambda)
             = \int_0^\infty e^{-t \sqrt{\lambda}} dE_{(U \circ W)f,(U \circ W)g}(\lambda) \\
        & \qquad  = \left\langle P_t^{\mathcal{L}} (U \circ W)f, (U \circ W)g \right\rangle_{L^2((0,T)\times \Omega,dt \times d\mu;X)},
        \quad f,g \in L^2((0,T)\times \Omega,dt \times d\overline{\mu};X),
    \end{align*}
    being $E$ and $\overline{E}$ the resolutions of the identity of $\mathcal{L}$ and $\overline{\mathcal{L}}$, respectively.
    Hence,
    \begin{align*}
       & \| K(\{P_u^{\overline{\mathcal{L}}}\}_{u>0}) (f)\|_{L^2((0,T)\times \Omega,dt \times d\overline{\mu};X)}
             = \| K(\{P_u^{\mathcal{L}}\}_{u>0}) ((U \circ W)f)\|_{L^2((0,T)\times \Omega,dt \times d\mu;X)} \\
       & \qquad \leq C \|(U \circ W)f\|_{L^2((0,T)\times \Omega,dt \times d\mu;X)}
           = C \| f \|_{L^2((0,T)\times \Omega,dt \times d\overline{\mu};X)}, \quad f \in L^2((0,T)\times \Omega,dt \times d\overline{\mu};X),
    \end{align*}
    because $U \circ W$ is an isometry from $L^2((0,T)\times \Omega,dt \times d\overline{\mu};X)$ into $L^2((0,T)\times \Omega,dt \times d\mu;X)$.
    The same argument allows us to prove the reverse direction.
\end{proof}

We would like to remark that the implication $(a)\Rightarrow(b)$ in Theorem~\ref{Th1} can be obtained
as a special case of stronger results about maximal regularity and
functional calculus (see \cite[Chapter 10]{KuWe}). However, the
important part of our Theorem~\ref{Th1} is just the other
direction, which give us new
characterizations of UMD Banach spaces. Furthermore, with our
method it doesn't cost more effort to prove one than two
directions of each implications in this theorem.

In Sections~\ref{sec:Bessel}, \ref{sec:Hermite} and \ref{sec:Laguerre} we prove Theorem \ref{Th1} in the Bessel, Hermite and Laguerre settings, respectively.

Throughout this paper by $C,c>0$ we always denote positive constants that can change in each occurrence.

%\newpage
%%%%%%%%%%%%%%%%%%%%%%%%%%%%%%%%%%%%%%%%%%%%%%%%%%%%%%%%%%%%%%%%%%%%%%%%%%%%%%%%%%%%%%%%%%%%%%%%%%%%%%%%%%%%%%%%%%%%
\section{Proof of Theorem~\ref{Th1} for Bessel type operators}\label{sec:Bessel}
%%%%%%%%%%%%%%%%%%%%%%%%%%%%%%%%%%%%%%%%%%%%%%%%%%%%%%%%%%%%%%%%%%%%%%%%%%%%%%%%%%%%%%%%%%%%%%%%%%%%%%%%%%%%%%%%%%%%

\subsection{Bessel operator $S_\alpha$}

We consider for every $\alpha>0$ the Bessel operator
$$S_\alpha
    =-x^{-\alpha}{d\over{dx}}x^{2\alpha}{d\over{dx}}x^{-\alpha}
    = - \frac{d^2}{dx^2} + \frac{\alpha(\alpha-1)}{x^2}, \quad x \in (0,\infty).$$
If $J_{\nu}$ denotes the Bessel function of first kind and order $\nu$, we have that
$$S_\alpha\left[\sqrt{xy}J_{\alpha -1/2}(xy)\right]=y^2\sqrt{xy}J_{\alpha -1/2}(xy), \quad x\in(0,\infty).$$

According to \cite[(16.4)]{MS} the Poisson kernel for $S_\alpha$ is the following
$$P_t^{S_\alpha}(x,y)={{2\alpha(xy)^{\alpha}t}\over{\pi}}\int_0^{\pi}{{{(\sin\theta)^{2\alpha -1}}\over{\left[|x-y|^2+t^2+2xy(1-\cos\theta)\right]^{\alpha +1}}}d\theta}, \quad t,x,y\in(0,\infty),$$
and the Poisson integral $P_t^{S_\alpha}(f)$ of $f\in L^p(0,\infty)$, $1\leq p\leq\infty$, is defined by
$$P_t^{S_\alpha}(f)(x)=\int_0^{\infty}{P_t^{S_\alpha}(x,y)f(y)dy}, \quad x\in (0,\infty), \ t>0.$$
The Bessel-Poisson semigroup $\{P_t^{S_\alpha}\}_{t>0}$ is positive and bounded in $L^p(0,\infty)$, $1\leq p\leq\infty$. Moreover, $\{P_t^{S_\alpha}\}_{t>0}$
is contractive in $L^p(0,\infty)$, $1\leq p\leq\infty$, when $\alpha >1$ (\cite[Proposition 6.1]{NoSt3}).
Harmonic analysis associated with Bessel operators was initiated by Muckenhoupt and Stein (\cite{MS}). In the last years, this Bessel harmonic
analysis has been developed in a Banach valued setting (see \cite{BCR1} and \cite{BFMT}).

We denote by $C_c^\infty(\mathcal{R})$ the space of smooth functions in $\mathcal{R}$ with compact support, where $\mathcal{R}=\R\setminus\{0\}$
or $\mathcal{R}=(0,\infty)$. It is well-known that $C_c^\infty(0,\infty)$ is a dense subspace of $L^2(0,\infty)$ and $C_c^\infty(\R\setminus \{0\})$
is dense in $L^2(\R)$.

We consider the operators $K(\{P^{S_\alpha}_u\}_{u>0})$ and $K(\{P_u\}_{u>0})$ defined, for each $t,x\in (0,\infty)$, by
$$K(\{P^{S_\alpha}_u\}_{u>0})(f)(t,x)
    =\int_0^t{\partial_tP^{S_\alpha}_{t-s}(f(s,\cdot))(x)ds}, \quad f \in C_c^\infty(0, T)\otimes C_c^\infty(0,\infty)\otimes X,$$
and for $t\in (0,\infty)$, $x\in\R$,
$$K(\{P_u\}_{u>0})(f)(t,x)
    =\int_0^t{\partial_tP_{t-s}(f(s,\cdot))(x)ds}, \quad f\in C_c^\infty(0, T)\otimes C_c^\infty(\R \setminus \{0\})\otimes X.$$
Note that, for every $f \in C_c^\infty(0, T)\otimes
C_c^\infty(0,\infty)\otimes X$,
$$\left( \int_0^t \partial_t P_{t-s}(f(s,\cdot)) ds \right)(x)
    = \int_0^t \partial_t P_{t-s}(f(s,\cdot))(x) ds, \quad \text{a.e. } x \in \R,$$
where the first integral is understood in the $L^2(\R,X)$-Bochner sense and the integral in the right hand side is understood in the $X$-Bochner sense.
Analogous properties hold when Hermite, Bessel and Laguerre Poisson semigroups replace the classical Poisson semigroup.

We will prove that $K(\{P^{S_\alpha}_u\}_{u>0})$ is bounded from
$L^2((0,T) \times (0,\infty),X)$ into itself if, and only if,
$K(\{P_u\}_{u>0})$ is bounded from $L^2((0,T)\times\R,X)$ into
itself. Then, by using Theorem \ref{ThA} we deduce
$(a)\Leftrightarrow(b)$ in Theorem \ref{Th1} for the Bessel
operator $S_\alpha$.

We define the operator $K_+(\{P_u\}_{u>0})$, for each $t,x\in (0,\infty)$,  as follows
$$K_+(\{P_u\}_{u>0})(f)(t,x)
    =\int_0^t {\int_0^\infty{\partial_t P_{t-s}(x-y)f(s,y)dy}ds},\quad f \in C_c^\infty(0, T)\otimes C_c^\infty(0,\infty)\otimes X.$$
In the first step we reduce the boundedness of $K(\{P_u\}_{u>0})$
in $L^2((0,T)\times\R,X)$ to the bounded\-ness of
$K_+(\{P_u\}_{u>0})$ in $L^2((0,T) \times (0,\infty),X)$.

\begin{Lem}\label{Lem2.1}
    Let $X$ be a Banach space and $T>0$. Then, $K(\{P_u\}_{u>0})$ is bounded in $L^2((0,T)\times\R,X)$ if, and only if,
    $K_+(\{P_u\}_{u>0})$ is bounded in $L^2((0,T) \times (0,\infty),X)$.
\end{Lem}

\begin{proof}
    Let $f\in C_c^\infty(0, T)\otimes C_c^\infty(0,\infty)\otimes X$. We defined for every $t \in (0,T)$,
    \begin{equation}\label{5.1}
        f_0(t,x)=\left\{\begin{array}{ll}
            f(t,x), & x>0, \\
            0, & x\leq 0.
            \end{array}\right.
    \end{equation}
    It is clear that $f_0\in C_c^\infty(0, T)\otimes C_c^\infty(\R \setminus \{0\})\otimes X$ and
    $$K_+(\{P_u\}_{u>0})(f)(t,x)
        =K(\{P_u\}_{u>0})(f_0)(t,x), \quad t,x\in (0,\infty).$$
    Then, if $K(\{P_u\}_{u>0})$ is bounded in $L^2((0,T)\times\R,X)$, $K_+(\{P_u\}_{u>0})$ is bounded in $L^2((0,T) \times (0,\infty),X)$.

    Suppose now that $K_+(\{P_u\}_{u>0})$ is bounded in $L^2((0,T) \times (0,\infty),X)$ and let
    $f \in C_c^\infty(0, T)\otimes C_c^\infty(\R \setminus \{0\})\otimes X$.
    We can write for every $t>0$ and $x\in\R$,
    \begin{align*}
        K(\{P_u\}_{u>0})(f)(t,x)
            = & \int_0^t{\int_{\R}{{\partial}_tP_{t-s}(x-y)f(s,y)dy}ds} \\
            = & \int_0^t{\int_0^\infty{{\partial}_tP_{t-s}(x-y)f(s,y)dy}ds}+ \int_0^t{\int_0^\infty{{\partial}_tP_{t-s}(x+y)f(s,-y)dy}ds},
    \end{align*}
    and
    \begin{align}\label{D1}
        & \int_0^{\infty}\int_{\R}  \|K(\{P_u\}_{u>0})(f)(t,x)\|_X^2dxdt
            = \int_0^{\infty}\int_0^{\infty}\|K(\{P_u\}_{u>0})(f)(t,x)\|_X^2dxdt \nonumber \\
        & \qquad \qquad + \int_0^{\infty}\int_0^{\infty}\|K(\{P_u\}_{u>0})(f)(t,-x)\|_X^2dxdt \nonumber \\
        & \leq 2\Big[\|K_+(\{P_u\}_{u>0})(f_+)\|_{L^2((0,T) \times (0,\infty),X)}^2+\|K_+(\{P_u\}_{u>0})(f_-)\|_{L^2((0,T) \times (0,\infty),X)}^2 \nonumber \\
        &\qquad \qquad +\|K_-(\{P_u\}_{u>0})(f_+)\|_{L^2((0,T) \times (0,\infty),X)}^2+\|K_-(\{P_u\}_{u>0})(f_-)\|_{L^2((0,T) \times (0,\infty),X)}^2\Big],
    \end{align}
    where $f_+(t,x)=f(t,x)$, $f_-(t,x)=f(t,-x)$, $t,x\in (0,\infty)$, and,
    $$K_-(\{P_u\}_{u>0})(g)(t,x)
        =\int_0^t{\int_0^\infty{{\partial}_tP_{t-s}(x+y)g(s,y)dy}ds}, \quad g \in C_c^\infty(0, T)\otimes C_c^\infty(0,\infty)\otimes X.$$

    The operator $K_-$ is bounded from $L^2((0,T) \times (0,\infty),X)$ into itself. Indeed, we have that
    $$|{\partial}_tP_{t-s}(x+y)|
        ={1\over{\pi}}\left|{{|t-s|^2-(x+y)^2}\over{(|t-s|^2+(x+y)^2)^2}}\right|\leq {1\over{\pi}}{1\over{|t-s|^2+(x+y)^2}},\quad s,t,x,y \in (0,\infty).$$
    Hence,
    \begin{align}\label{D.2}
        & \|K_-(\{P_u\}_{u>0})(g)(t,x)\|_X\leq{1\over{\pi}}\int_0^t{\int_0^{\infty}{{1\over{|t-s|^2+(x+y)^2}}\|g(s,y)\|_Xdy}ds} \nonumber \\
        & \qquad \leq{1\over{\pi}}\int_0^{\infty}{\left(\int_0^x{{1\over{|t-s|^2+x^2}}\|g(s,y)\|_Xdy}+\int_x^{\infty}{{1\over{|t-s|^2+y^2}}\|g(s,y)\|_Xdy}\right)ds} \nonumber\\
        & \qquad \leq{1\over{\pi}}\left({1\over x}\int_0^x{P_*(\|\tilde{g_0}(\cdot,y)\|_X)(t)dy}+\int_x^{\infty}{{1\over y}P_*(\|\tilde{g_0}(\cdot,y)\|_X)(t)dy}\right),\quad t,x\in (0,\infty),
    \end{align}
    where, for every $y \in (0,\infty)$,
    \begin{equation}\label{ext}
        \tilde{g_0}(t,y)=
            \left\{\begin{array}{ll}
                g_0(t,y), & t \in (0,T), \\
                0, & t \notin (0,T),
            \end{array}\right.
    \end{equation}
    and $P_*$
    represents the maximal operator defined by
    $$P_*(h)=\sup_{u>0}|P_u(h)|, \quad h\in L^2(\R).$$
    According to \cite[p. 244, $(9.9.1)$ and $(9.9.2)$]{HLP}, the Hardy operators $H_0$ and $H_{\infty}$ given by
    $$H_0(h)(x)={1\over x}\int_0^x{h(y)dy}, \quad x\in (0,\infty),$$
    and
    $$H_{\infty}(h)(x)=\int_x^{\infty}{{h(y)\over y}dy}, \quad x\in (0,\infty),$$
    are bounded in $L^2(0,\infty)$. Also, the maximal operator $P_*$ is bounded in $L^2(\R)$.
    Then, from \eqref{D.2} we deduce that $K_-$ is bounded in $L^2((0,T) \times (0,\infty),X)$.

    By \eqref{D1} it follows that $K(\{P_u\}_{u>0})$ is bounded from $L^2((0,T)\times\R,X)$ into itself provided that $K_+(\{P_u\}_{u>0})$
    is bounded from $L^2((0,T) \times (0,\infty),X)$ into itself.
\end{proof}

We now consider, for every $t,x \in (0,\infty)$ and $f \in
C_c^\infty(0, T)\otimes C_c^\infty(0,\infty)\otimes X$,
the following ``local'' and ``global'' operators \\

$\displaystyle \bullet \ K_+^{loc}(\{P_u\}_{u>0})f(t,x)
    =\int_0^t{\int_{x/2}^{2x}{{\partial}_tP_{t-s}(x-y)f(s,y)dyds}}, $ \\

$\displaystyle \bullet \ K^{glob}_+(\{P_u\}_{u>0})f(t,x)
    =K_+(\{P_u\}_{u>0})f(t,x)-K_+^{loc}(\{P_u\}_{u>0})f(t,x),$ \\

$\displaystyle \bullet \ K^{loc}(\{P_u^{S_\alpha}\}_{u>0})f(t,x)
    =\int_0^t{\int_{x/2}^{2x}{{\partial}_tP_{t-s}^{S_\alpha}(x,y)f(s,y)dyds}}, $ \\
\noindent and

$\displaystyle \bullet \ K^{glob}(\{P_u^{S_\alpha}\}_{u>0})f(t,x)
    =K(\{P_u^{S_\alpha}\}_{u>0})f(t,x)-K^{loc}(\{P_u^{S_\alpha}\}_{u>0})f(t,x).$\\
We are going to see that the ``global'' operators are bounded in
$L^2((0,T) \times (0,\infty),X)$.

\begin{Lem}\label{Lem2.2}
    Let $X$ be a Banach space and $T>0$. Then, the operators $K^{glob}_+(\{P_u\}_{u>0})$ and $K^{glob}(\{P_u^{S_\alpha}\}_{u>0})$
    are bounded from $L^2((0,T) \times (0,\infty),X)$ into itself.
\end{Lem}

\begin{proof}
    Let $f \in C_c^\infty(0, T)\otimes C_c^\infty(0,\infty)\otimes X$.
    The argument used in the proof of Lemma~\ref{Lem2.1} allows us to get, for every $t,x \in (0,\infty)$,
    \begin{align*}
        & \|K_+^{glob}(\{P_u\}_{u>0})(f)(t,x)\|_X
            \leq C \left({1\over x}\int_0^{x/2}{P_*(\|\tilde f_0(\cdot ,y)\|_X)(t)dy}
            +\int_{2x}^{\infty}{{1\over y}P_*(\|\tilde f_0(\cdot ,y)\|_X)(t)dy}\right),
    \end{align*}
    where the extension $\tilde f_0$ of $f$ is defined as in \eqref{ext}.
    Then, $K_+^{glob}$ is bounded in $L^2((0,T) \times (0,\infty),X)$.

    On the other hand, we have that
    \begin{align*}
    \partial_tP_{t-s}^{S_\alpha}(x,y)
        = & {{2\alpha(xy)^{\alpha}}\over{\pi}}\left\{\int_0^{\pi}{{{(\sin\theta)^{2\alpha -1}}\over{\left[|x-y|^2+|t-s|^2+2xy(1-\cos\theta)\right]^{\alpha +1}}}d\theta} \right. \\
        & \left. -2(\alpha +1)|t-s|^2\int_0^{\pi}{{{(\sin\theta)^{2\alpha -1}}\over{\left[|x-y|^2+|t-s|^2+2xy(1-\cos\theta)\right]^{\alpha +2}}}d\theta}\right\}, \ s,t,x,y \in (0,\infty).
    \end{align*}

    According to \cite[p. 86, $(b)$]{MS} we obtain
    \begin{align*}
        \left|\partial_tP_{t-s}^{S_\alpha}(x,y)\right|
        \leq & C(xy)^{\alpha}\int_0^{\pi}{{{(\sin\theta)^{2\alpha -1}}\over{\left[|x-y|^2+|t-s|^2+2xy(1-\cos\theta)\right]^{\alpha +1}}}d\theta} \\
        \leq & C{1\over{|x-y|^2+|t-s|^2}}, \quad s,t,x,y \in (0,\infty).
    \end{align*}

    Then, as above we can prove that $K^{glob}(\{P_u^{S_\alpha}\}_{u>0})$ is bounded in $L^2((0,T) \times (0,\infty),X)$.
\end{proof}

Hence, we will finish the proof once we show that
$K_+^{loc}(\{P_u\}_{u>0})$ is bounded in $L^2((0,T) \times
(0,\infty),X)$ if, and only if,
$K^{loc}(\{P_u^{S_\alpha}\}_{u>0})$ is bounded in $L^2((0,T)
\times (0,\infty),X)$. We collect this property in the next Lemma.

\begin{Lem}\label{Lem2.3}
    Let $X$ be a Banach space and $T>0$. Then, the operator $D^{loc}$ given by
    $$D^{loc}=K^{loc}(\{P_u^{S_\alpha}\}_{u>0})-K_+^{loc}(\{P_u\}_{u>0}),$$
    is bounded in $L^2((0,T) \times (0,\infty),X)$.
\end{Lem}

\begin{proof}
    Let $f \in C_c^\infty(0, T)\otimes C_c^\infty(0,\infty)\otimes X$.
    We split the operator $K^{loc}(\{P_u^{S_\alpha}\}_{u>0})$ as follows
    $$K^{loc}(\{P_u^{S_\alpha}\}_{u>0})=K_1^{loc}(\{P_u^{S_\alpha}\}_{u>0})+K_2^{loc}(\{P_u^{S_\alpha}\}_{u>0}),$$
    being,
    $$K_1^{loc}(\{P_u^{S_\alpha}\}_{u>0})(f)(t,x)
        =\int_0^t{\int_{x/2}^{2x}{\partial_tP_{t-s}^{S_\alpha,1}(x,y)f(s,y)dy}ds}, \quad t,x\in (0,\infty),$$
    and
    $$P_t^{S_\alpha,1}(x,y)
        ={{2\alpha(xy)^{\alpha}}\over{\pi}}\int_0^{\pi / 2}{{{t(\sin\theta)^{2\alpha -1}}\over{\left[|x-y|^2+t^2+2xy(1-\cos\theta)\right]^{\alpha +1}}}d\theta}, \quad t,x,y\in (0,\infty).$$
    By defining
    $$P_t^{S_\alpha,2}(x,y)
        =P_t^{S_\alpha}(x,y)-P_t^{S_\alpha,1}(x,y), \quad t,x,y\in (0,\infty),$$
    we can write
    \begin{align*}
        \left|\partial_tP_{t-s}^{S_\alpha,2}(x,y)\right|
            \leq & C(xy)^{\alpha}\int_{\pi / 2}^{\pi}{{{(\sin\theta)^{2\alpha -1}}\over{\left[|x-y|^2+|t-s|^2+2xy(1-\cos\theta)\right]^{\alpha +1}}}d\theta} \\
            \leq & C{1\over{|t-s|^2+x^2}}, \quad s,t,x \in (0,\infty), \ {x/2}<y<2x.
    \end{align*}

    Hence, we get
    $$\|K_2^{loc}(\{P_u^{S_\alpha}\}_{u>0})(f)(t,x)\|_X
        \leq {C\over x}\int_{x/2}^{2x}{P_*(\|\tilde f_0(\cdot,y)\|_X)(t)dy,\quad t,x\in (0,\infty)}.$$
    Then, since the operator $\mathbb{L}$ defined by
    $$\mathbb{L}(h)(x)={1\over x}\int_{x/2}^{2x}{h(y)dy},\quad x\in (0,\infty),$$
    is bounded in $L^2(0,\infty)$, we conclude that $K_2^{loc}(\{P_u^{S_\alpha}\}_{u>0})$ is bounded in $L^2((0,T) \times (0,\infty),X)$.

    We consider the kernels
    $$P_t^{S_\alpha  ,1,1}(x,y)
        ={{2\alpha(xy)^{\alpha}}\over{\pi}}\int_0^{\pi / 2}{{{t\theta^{2\alpha -1}}\over{\left[|x-y|^2+t^2+2xy(1-\cos\theta)\right]^{\alpha +1}}}d\theta}, \quad t,x,y\in (0,\infty),$$
    $$P_t^{S_\alpha  ,1,2}(x,y)
        ={{2\alpha(xy)^{\alpha}}\over{\pi}}\int_0^{\pi / 2}{{{t\theta^{2\alpha -1}}\over{\left[|x-y|^2+t^2+xy\theta^2\right]^{\alpha +1}}}d\theta}, \quad t,x,y\in (0,\infty).$$
    and we define the operators
    $$D_1^{loc}(f)(t,x)
        =\int_0^t{\int_{x/2}^{2x}{\partial_t\left[P_{t-s}^{S_\alpha,1}(x,y)-P_{t-s}^{S_\alpha  ,1,1}(x,y)\right]f(s,y)dy}ds}, \quad t,x\in (0,\infty)$$
    and
    $$D_2^{loc}(f)(t,x)
        =\int_0^t{\int_{x/2}^{2x}{\partial_t\left[P_{t-s}^{S_\alpha  ,1,1}(x,y)-P_{t-s}^{S_\alpha  ,1,2}(x,y)\right]f(s,y)dy}ds}, \quad t,x\in (0,\infty).$$

    In order to estimate the kernels of $D_1^{loc}$ and $D_2^{loc}$ we proceed as in \cite[pp. 482--485]{BCFR1}.
    By using mean value theorem, we get for every $s,t,x \in (0,\infty)$ and $x/2<y<2x$,
    \begin{align*}
        & \left|\partial_t \left[P_{t-s}^{S_\alpha,1}(x,y)-P_{t-s}^{S_\alpha  ,1,1}(x,y)\right]\right|
            +\left| \partial_t \left[P_{t-s}^{S_\alpha  ,1,1}(x,y)-P_{t-s}^{S_\alpha  ,1,2}(x,y)\right]\right| \\
        & \qquad \qquad \leq C(xy)^{\alpha}\int_0^{\pi / 2}{{{\theta^{2\alpha +1}}\over{\left[|x-y|^2+|t-s|^2+xy\theta^2\right]^{\alpha +1}}}d\theta} \\
        & \qquad \qquad \leq C{{(xy)^{\alpha}}\over{\left[|x-y|^2+|t-s|^2+xy\right]^{\alpha +1}}}\left[1+\log_+\left({{xy}\over{|x-y|^2+|t-s|^2}}\right)\right].
    \end{align*}
    Hence, for $j=1,2$,
    \begin{align*}
        \|D_j^{loc}(f)(t,x)\|_X
            \leq &  C\int_0^t{\int_{x/2}^{2x}{{1\over{x^2+|t-s|^2}}\left(1+\log_+{{x}\over{|x-y|}}\right)\|f(s,y)\|_Xdy}ds} \\
            \leq & C{1\over x}\int_{x/2}^{2x}{\left(1+\log_+{{x}\over{|x-y|}}\right)P_*(\|\tilde f_0(\cdot,y)\|_X)(t)dy}, \quad t,x\in (0,\infty).
    \end{align*}
    Since
    $$ {1\over x}\int_{x/2}^{2x}{\left(1+\log_+{{x}\over{|x-y|}}\right)dy}=\int_{1/ 2}^2{\left(1+\log_+{1\over{|1-z|}}\right)dz}, \quad x\in (0,\infty),$$
    Jensen's inequality allows us to show that the operator $\mathcal{A}$ defined by
    $$\mathcal{A}(g)(x)
        ={1\over x}\int_{x/2}^{2x}{\left(1+\log_+{{x}\over{|x-y|}}\right)g(y)dy}, \quad x\in (0,\infty),$$
    is bounded from $L^2(0,\infty)$ into itself. Hence, the operators $D_1^{loc}$ and $D_2^{loc}$ are bounded in $L^2((0,T) \times (0,\infty),X)$.

    Finally, we consider the operator $D_3^{loc}$ as follows
    $$D_3^{loc}(f)(t,x)
        =\int_0^t{\int_{x/2}^{2x}{\partial_t\left[P_{t-s}^{S_\alpha  ,1,2}(x,y)-P_{t-s}(x-y)\right]f(s,y)dy}ds}, \quad t,x\in (0,\infty).$$
    As in \cite[p. 486]{BCFR1} we can write, for every $t,x,y \in (0,\infty)$,
    $$P_t^{S_\alpha,1,2}(x,y)
        =P_t(x-y)-\frac{2\alpha (xy)^\alpha}{\pi}\int_{\pi / 2}^{\infty}{{{t{\theta}^{2\alpha -1}}\over{\left[|x-y|^2+t^2+xy\theta^2\right]^{\alpha +1}}}d\theta}.$$
    Then, we get
    $$D_3^{loc}(f)(t,x)
        =\int_0^t{\int_{x/2}^{2x}{\partial_tP_{t-s}^{S_\alpha  ,1,3}(x,y)f(s,y)dy}ds}, \quad t,x\in (0,\infty),$$
    where
    $$P_{t}^{S_\alpha  ,1,3}(x,y)
        =-{{2\alpha(xy)^{\alpha}}\over{\pi}}\int_{\pi / 2}^{\infty}{{{{t\theta}^{2\alpha -1}}\over{\left[|x-y|^2+t^2+xy\theta^2\right]^{\alpha +1}}}d\theta}, \quad t,x,y\in (0,\infty).$$

    We have that
    \begin{align*}
        & \left|{\partial}_tP_{t-s}^{S_\alpha  ,1,3}(x,y)\right|\leq C\int_{\pi / 2}^{\infty}{{{(xy)^{\alpha}{\theta}^{2\alpha -1}}\over{\left[|x-y|^2+|t-s|^2+xy\theta^2\right]^{\alpha +1}}}d\theta} \\
        & \qquad \leq C(xy)^{\alpha}\left[{1\over{\left[|x-y|^2+|t-s|^2+xy\pi^2/4\right]^{\alpha +1}}}+\int_{\pi^2/4}^{\infty}{{{xy{u}^{\alpha}}\over{\left[|x-y|^2+|t-s|^2+xyu\right]^{\alpha +2}}}du}\right] \\
        & \qquad \leq C(xy)^{\alpha}\left[{1\over{\left[|t-s|^2+xy\right]^{\alpha +1}}}+(xy)^{-\alpha +1}\int_{\pi^2/4}^{\infty}{{1\over{\left[|t-s|^2+xyu\right]^2}}du}\right] \\
        & \qquad \leq C{1\over{|t-s|^2+x^2}}, \quad s,t,x \in (0,\infty) \mbox{ and } {x/2}<y<2x.
    \end{align*}

    We can write
    \begin{align*}
        \|D_3^{loc}(f)(t,x)\|_X
            & \leq C\int_0^t{\int_{x/2}^{2x}{{{\|f(s,y)\|_X}\over{|t-s|^2+x^2}}dy}ds}
            \leq C{1\over x}\int_{x/2}^{2x}{P_*(\|\tilde f_0(\cdot,y)\|_X)(t)dy}, \ t,x\in(0,\infty).
    \end{align*}

    Hence, $D_3^{loc}$ is bounded from $L^2((0,T) \times (0,\infty),X)$ into itself. We conclude that $D^{loc}$ is bounded in $L^2((0,T) \times (0,\infty),X)$.
\end{proof}

\subsection{Bessel operator $\Delta_\alpha$}

We just have to use Lemma~\ref{LemTransfer}, taking $\overline{\mathcal{L}}=\Delta_\alpha$, $\mathcal{L}=S_\alpha$, $M(x)=x^\alpha$
and $h(x)=x$, so
$$\Delta_\alpha = (U \circ W)^{-1} \circ S_\alpha \circ (U \circ W).$$

%\newpage
%%%%%%%%%%%%%%%%%%%%%%%%%%%%%%%%%%%%%%%%%%%%%%%%%%%%%%%%%%%%%%%%%%%%%%%%%%%%%%%%%%%%%%%%%%%%%%%%%%%%%%%%%%%%%%%%%%%%
\section{Proof of Theorem~\ref{Th1} for Hermite type operators}\label{sec:Hermite}
%%%%%%%%%%%%%%%%%%%%%%%%%%%%%%%%%%%%%%%%%%%%%%%%%%%%%%%%%%%%%%%%%%%%%%%%%%%%%%%%%%%%%%%%%%%%%%%%%%%%%%%%%%%%%%%%%%%%

\subsection{Hermite operator $H$}

The Hermite operator on $\R$ is defined by $\displaystyle H=-{1\over 2}({{d^2}\over{dx^2}}-x^2)$. We have that, for every $k\in\N$,
$$Hh_k=\left(k+{1\over 2}\right)h_k,$$
where $h_k$ denotes the $k$-th Hermite function given by
$$h_k(x)=(\sqrt{k}2^kk!)^{-{1/ 2}}H_k(x)e^{-{{x^2}/ 2}}, \quad x\in\R,$$
and $H_k$ represents the $k$-th Hermite polynomial \cite[pp. 105--106]{Sz}.
The operator $-H$ gene\-rates the positive semigroup of contractions
$\{W_t^H\}_{t>0}$ in $L^p(\R)$, being, for every $f\in L^p(\R)$, $1~\leq~p~\leq~\infty$,
$$W_t^H(f)(x)=\int_\R{W_t^H(x,y)f(y)dy}, \quad x\in\R \mbox{ and } t>0,$$
and for each $x,y \in \R$, $t>0$,
$$W_t^H(x,y)
    =\left({{e^{-t}}\over{\pi(1-e^{-2t})}}\right)^{1/ 2}
        \exp \left[-{1\over 2}{{1+e^{-2t}}\over{1-e^{-2t}}}(x^2+y^2)+{{2e^{-t}}\over{1-e^{-2t}}}xy\right].$$

The Poisson semigroup $\{P_t^H\}_{t>0}$ associated with the Hermite operator is defined, for every $f\in L^p(\R)$, $1\leq p\leq\infty$,  as
$$P_t^H(f)(x)
    =\int_\R{P_t^H(x,y)f(y)dy},\quad x\in\R \mbox{ and } t>0,$$
where, by using the subordination formula we can write
$$P_t^H(x,y)
    =\frac{t}{\sqrt{4\pi}}\int_0^{\infty} \frac{e^{-t^2/4u}}{u^{3/2}} W_u^H(x,y)du,\quad x,y\in\R \mbox{ and } t>0.$$

Muckenhoupt studied in \cite{Mu2} and \cite{Mu1} the harmonic analysis for the Hermite operator in dimension one.
In the last years, Thangavelu \cite{Th}, Stempak and Torrea \cite{ST1} and \cite{ST2} have investigated it in higher dimensions.
Banach valued harmonic analysis operators in the Hermite context have been studied in \cite{AT} and \cite{BFRST1}.

We consider the operator defined by
$$K\left(\{P_u^H\}_{u>0}\right)(f)(t,x)
    = \int_0^t \partial_t P_{t-s}^H(f(s,\cdot))(x)ds, \quad x \in \R \text{ and } t>0, $$
for every $f \in C^\infty_c(0,T) \otimes C_c^\infty(\R)\otimes X$.

We denote by $\rho$ the following function
$$ \rho(x)=\left\{
        \begin{array}{ll}
            1/2, &|x| \leq 1,\\
            \dfrac{1}{1+|x|}, & |x| > 1.
        \end{array}\right. $$
This function $\rho$ plays an important role in the study of harmonic analysis associated with the Hermite operator
(see \cite{BCFST} and \cite{DZ1}).

We split the operator $K\left(\{P_u^H\}_{u>0}\right)$ as follows
$$K\left(\{P_u^H\}_{u>0}\right)
    = K^{loc}\left(\{P_u^H\}_{u>0}\right) + K^{glob}\left(\{P_u^H\}_{u>0}\right),$$
where
$$K^{loc}\left(\{P_u^H\}_{u>0}\right)(f)(t,x)
    = \int_0^t \int_{|x-y|<\rho(x)} \partial_t P_{t-s}^H(x,y) f(s,y) dy ds, \quad x \in \R \text{ and } t>0,$$
for every $f \in C^\infty_c(0,T) \otimes C_c^\infty(\R)\otimes X$.

In a similar way we decompose the operator
$K\left(\{P_u\}_{u>0}\right)$ on $C^\infty_c(0,T) \otimes C_c^\infty(\R)\otimes X$ by
$$K\left(\{P_u\}_{u>0}\right)
    = K^{loc}\left(\{P_u\}_{u>0}\right) + K^{glob}\left(\{P_u\}_{u>0}\right).$$

In a first step, we are going to see that the operator
$K^{glob}\left(\{P_u^H\}_{u>0}\right)$ is bounded in
$L^2((0,T)\times \R, X)$.

\begin{Lem}\label{Lem3.1}
    Let $X$ be a Banach space and $T>0$. Then, the operator $K^{glob}\left(\{P_u^H\}_{u>0}\right)$ is bounded from
    $L^2((0,T)\times \R, X)$ into itself.
\end{Lem}

\begin{proof}
    Let $f \in C^\infty_c(0,T) \otimes C_c^\infty(\R)\otimes X$. We have that
    \begin{align*}
        \left| \partial_t P_{t-s}^H(x,y)  \right|
            \leq & C \int_0^\infty \frac{e^{-|t-s|^2/4u}}{u^{3/2}} \left( 1 + \frac{|t-s|^2}{u} \right) W_u^H(x,y) du \\
            \leq & C \int_0^\infty \frac{e^{-c|t-s|^2/u}}{u^{3/2}} W_u^H(x,y) du, \quad x,y \in \R \text{ and } s,t \in (0,\infty).
    \end{align*}
    Then,
    \begin{align}\label{3.1}
        & \left\| K^{glob}\left(\{P_u^H\}_{u>0}\right)(f)(t,x) \right\|_X \nonumber \\
        & \qquad \qquad \leq C \int_0^t \int_{|x-y| \geq \rho(x)} \left( \int_0^{\rho(x)^2} + \int_{\rho(x)^2}^\infty \right) \frac{e^{-c|t-s|^2/u}}{u^{3/2}} W_u^H(x,y) du
                                    \|f(s,y)\|_X dy ds \nonumber \\
        & \qquad \qquad = C \left[ K^{glob}_{H,1}(\|f\|_X)(t,x) + K^{glob}_{H,2}(\|f\|_X)(t,x)\right], \quad x \in \R \text{ and } t \in (0,\infty).
    \end{align}
    Let $h \in L^2((0,T)\times \R)$. We detone by $h_0$ the following extension of $h$
    $$h_0(t,x)
        = \left\{
            \begin{array}{ll}
                h(t,x), & t \in (0,T),  \\
                0, & t \notin (0,T),
            \end{array} \right.$$
    for every $x \in \R$.
    Since
    \begin{equation}\label{10.1}
        W_t^H(x,y)
            \leq C e^{-ct} W_{ct}(x-y), \quad x,y \in \R \text{ and } t>0,
    \end{equation}
    where $W_t(z)=e^{-|z|^2/4t}/\sqrt{4\pi t}$, $z \in \R$ and $t>0$, represents the classical heat semigroup, we can write
    \begin{align*}
        \left| K^{glob}_{H,1}(h)(t,x) \right|
            \leq & C \int_0^t \int_{|x-y| \geq \rho(x)} \int_0^{\rho(x)^2} \frac{e^{-c|t-s|^2/u}}{u^{3/2}} \frac{e^{-c|x-y|^2/u}}{u^{1/2}}  |h(s,y)|du dy ds \\
            \leq & C \int_{|x-y| \geq \rho(x)} \int_0^{\rho(x)^2}  \frac{e^{-c|x-y|^2/u}}{u^{3/2}} W_*(|h_0(\cdot,y)|)(t) du dy ,
            \quad x \in \R \text{ and } t >0,
    \end{align*}
    where $W_*(g)=\sup_{u>0} |W_u(g)|$. We get
    \begin{align*}
    & \left| K^{glob}_{H,1}(h)(t,x) \right|
            \leq  C  \int_0^{\rho(x)^2} \int_{|x-y| \geq \rho(x)}  \frac{e^{-c \rho(x)^2/u}}{u} \frac{e^{-c |x-y|^2/u}}{u^{1/2}}  W_*(|h_0(\cdot,y)|)(t) dy du\\
    & \qquad \leq  C  \left(\int_0^{\rho(x)^2} \frac{du}{\rho(x)^2} \right) W_*\left[W_*(|h_0|)(t)\right](x)
            =  C  W_*\left[W_*(|h_0|)(t)\right](x), \quad x \in \R \text{ and } t >0.
    \end{align*}
    Since the maximal operator $W_*$ is bounded in $L^2(\R)$ we deduce that the operator $K^{glob}_{H,1}$ is bounded from
    $L^2((0,T)\times \R)$ into itself.

    In order to show that $K^{glob}_{H,2}$ is bounded in $L^2((0,T)\times \R)$ we will use the following estimate
    \begin{equation}\label{3.2}
        W_t^H(x,y)
            \leq C \frac{e^{-c|x-y|^2/t}}{\sqrt{t}} \left( \frac{\rho(x)^2}{t} \right)^2, \quad x,y \in \R \text{ and } t>0.
    \end{equation}
    Indeed, we have that
    $$|x-y|^2+|x+y|^2 \geq x^2 \geq \frac{C}{\rho(x)^2}, \quad |x| > 1.$$
    Then, since $\rho(x)=1/2$, $|x| \leq 1$,
    \begin{align*}
        W_t^H(x,y)
            \leq C e^{-c(t+|x+y|^2+|x-y|^2)}
            \leq C \frac{e^{-c|x-y|^2/t}}{\sqrt{t}} \left( \frac{\rho(x)^2}{t} \right)^2, \quad x,y \in \R \text{ and } t>1.
    \end{align*}
    Suppose now that $0<t \leq 1$. If $|x| \leq 1$ and $y \in \R$,
    \begin{align*}
        W_t^H(x,y)
            \leq C \frac{e^{-c(t+|x-y|^2/t)}}{\sqrt{t}}
            \leq C \frac{e^{-c|x-y|^2/t}}{\sqrt{t}} \left( \frac{\rho(x)^2}{t} \right)^2.
    \end{align*}
    Moreover, if $|x|>1$ and $y \in \R$,
    \begin{align*}
        W_t^H(x,y)
            \leq & C \frac{e^{-c(t+|x-y|^2/t+t|x+y|^2)}}{\sqrt{t}}
            \leq C \frac{e^{-c|x-y|^2/t}e^{-ctx^2}}{\sqrt{t}}
            \leq C \frac{e^{-c |x-y|^2/t} e^{-ct/\rho(x)^2} }{\sqrt{t}} \\
            \leq & C \frac{e^{-c|x-y|^2/t}}{\sqrt{t}} \left( \frac{\rho(x)^2}{t} \right)^2, \quad xy \geq 0,
    \end{align*}
    and
    \begin{align*}
        W_t^H(x,y)
            \leq & C \frac{e^{-c(t+|x-y|^2/t)}}{\sqrt{t}}
            \leq C \frac{e^{-ct} e^{-c|x-y|^2/t}e^{-cx^2/t}}{\sqrt{t}}
            \leq C \frac{e^{-c|x-y|^2/t}}{\sqrt{t}} \left( \frac{1}{tx^2} \right)^2 \\
            \leq & C \frac{e^{-c|x-y|^2/t}}{\sqrt{t}} \left( \frac{\rho(x)^2}{t} \right)^2, \quad xy \leq  0.
    \end{align*}
    Thus, \eqref{3.2} is established.

    We can write
    \begin{align*}
        \left| K^{glob}_{H,2}(h)(t,x) \right|
            \leq & C \int_0^t \int_{|x-y| \geq \rho(x)} \int_{\rho(x)^2}^\infty
                        \frac{e^{-c|t-s|^2/u}}{u^{1/2}} \frac{e^{-c|x-y|^2/u}}{u^{3/2}} \frac{\rho(x)^4}{u^2}  |h(s,y)| dudy ds   \\
            \leq & C \int_{|x-y| \geq \rho(x)} \left(\int_{\rho(x)^2}^\infty \frac{\rho(x)^4}{u^{7/2}}e^{-c|x-y|^2/u}  du \right) W_*(|h_0(\cdot,y)|)(t) dy \\
            \leq & C \int_{|x-y| \geq \rho(x)} \left( \rho(x)^6 \int_{1}^\infty \frac{e^{-c|x-y|^2/v\rho(x)^2}}{(v\rho(x)^2)^{7/2}}  dv\right) W_*(|h_0(\cdot,y)|)(t) dy \\
            \leq & C \int_{|x-y| \geq \rho(x)} \left( \frac{1}{\rho(x)} \int_{1}^\infty  \left( \frac{\rho(x)}{|x-y|} \right)^3 \frac{dv}{v^2} \right) W_*(|h_0(\cdot,y)|)(t) dy \\
            \leq & C \frac{1}{\rho(x)} \int_{|x-y| \geq \rho(x)} \left( \frac{\rho(x)}{|x-y|} \right)^3 W_*(|h_0(\cdot,y)|)(t) dy \\
            = & C \frac{1}{\rho(x)} \sum_{k=0}^\infty \int_{2^k \rho(x) \leq |x-y| < 2^{k+1}\rho(x)}\left( \frac{\rho(x)}{|x-y|} \right)^3 W_*(|h_0(\cdot,y)|)(t) dy \\
            \leq & C \sum_{k=0}^\infty \frac{1}{2^{3k}\rho(x)}  \int_{|x-y| < 2^{k+1}\rho(x)}W_*(|h_0(\cdot,y)|)(t) dy \\
            \leq & C \M\left[ W_*(|h_0|)(t) \right](x), \quad x \in \R \text{ and } t >0.
    \end{align*}
    Here $\M$ represents the Hardy-Littlewood maximal function. Since $W_*$ and $\M$ are bounded operators in
    $L^2(\R)$, it follows that the operator $K^{glob}_{H,2}$ is bounded from $L^2((0,T)\times \R)$
    into itself.

    From \eqref{3.1} we conclude that $K^{glob}(\{P^H_u\}_{u>0})$ is bounded from $L^2((0,T)\times \R,X)$ into itself.
\end{proof}

In the second step we show that the operator $\mathcal{D}^{loc}$ defined by
$$\mathcal{D}^{loc}
    = K^{loc}(\{P^H_u\}_{u>0}) - K^{loc}(\{P_u\}_{u>0})$$
is bounded in $L^2((0,T)\times \R,X)$.

\begin{Lem}\label{Lem3.2}
    Let $X$ be a Banach space and $T>0$. Then, the operator $\mathcal{D}^{loc}$ is bounded from
    $L^2((0,T)\times~\R,~X)$ into itself.
\end{Lem}

\begin{proof}
    Let $f \in C^\infty_c(0,T) \otimes C_c^\infty(\R)\otimes X$. The subordination formula and \eqref{10.1} allow us to write
    \begin{align*}
        \left\| \mathcal{D}^{loc}(f)(t,x) \right\|_X
            \leq &C \Big(\int_0^t \int_{|x-y| < \rho(x)} \int_0^{\rho(x)^2} \frac{e^{-c|t-s|^2/u}}{u^{3/2}}
                \left| W_u^H(x,y)-W_u(x-y) \right|  \|f(s,y)\|_X du dy ds  \\
        & + \int_0^t \int_{|x-y| < \rho(x)} \int_{\rho(x)^2}^\infty \frac{e^{-c|t-s|^2/u}}{u^{3/2}} W_{cu}(x-y)   \|f(s,y)\|_X du dy ds \Big)  \\
        = & C \left[ \mathcal{D}_1^{loc}(\|f\|_X)(t,x) + \mathcal{D}_2^{loc}(\|f\|_X)(t,x)\right], \quad x \in \R \text{ and } t \in (0,\infty).
    \end{align*}

    We now prove that $\mathcal{D}_j^{loc}$, $j=1,2$, is a bounded operator from $L^2((0,T)\times \R)$ into itself. Let
    $h \in L^2((0,T)\times \R)$. We have that
    \begin{align*}
        \left| \mathcal{D}^{loc}_{2}(h)(t,x) \right|
            \leq & C \int_{|x-y| < \rho(x)} \left(\int_{\rho(x)^2}^\infty \frac{du}{u^{3/2}} \right) W_*(|h_0(\cdot,y)|)(t) dy    \\
            \leq & \frac{C}{\rho(x)}  \int_{|x-y| < \rho(x)}W_*(|h_0(\cdot,y)|)(t) dy \\
            \leq & C \M\left[ W_*(|h_0|)(t) \right](x), \quad x \in \R \text{ and } t >0.
    \end{align*}
    Hence, $\mathcal{D}^{loc}_{2}$ is bounded in $L^2((0,T)\times \R)$.

    In order to analyze $\mathcal{D}^{loc}_{1}$ we need to show that
    \begin{equation}\label{3.3}
        \left| W_t^H(x,y)-W_t(x-y)\right|
            \leq C \frac{t}{\rho(x)^2} W_{ct}(x-y), \quad |x-y|<\rho(x) \text{ and } 0<t<\rho(x)^2.
    \end{equation}
    Indeed, by using Kato-Trotter's formula and \eqref{10.1} we get, for every $x,y \in \R$ and t>0,
    \begin{align*}
        & \left| W_t^H(x,y)-W_t(x-y)\right|
            =  \int_0^t \int_\R W_s(x-z)  W_{t-s}^H(z,y) z^2dz ds \\
        & \qquad \qquad    \leq  C \int_0^t \int_\R W_s(x-z)  W_{c(t-s)}(z-y) z^2dz ds \\
        & \qquad \qquad    \leq  C \left( \int_0^{t/2} \int_\R W_s(x-z)  W_{c(t-s)}(z-y) z^2dz ds + \int_0^{t/2} \int_\R W_{t-s}(x-z)  W_{cs}(z-y) z^2dz ds \right).
    \end{align*}
    On the other hand, we have that
    \begin{align*}
        \int_\R e^{-c|x-z|^2/s}z^2 dz
            = & \sqrt{s} \int_\R e^{-w^2} (x^2+sw^2-2xw\sqrt{s}) dw \\
            \leq & C \sqrt{s} \int_\R e^{-w^2} (x^2+w^2+xw) dw, \quad x \in \R, \ 0<s<1.
    \end{align*}
    If $|x|\leq 1$, then
    $$\int_\R e^{-w^2} (x^2+w^2+xw) dw
        \leq \int_\R e^{-w^2} (1+w^2+w) dw
        \leq C
        \leq \frac{C}{\rho(x)^2},$$
    and, when $|x| \geq 1$,
    \begin{align*}
        \int_\R e^{-w^2} (x^2+w^2+xw) dw
            \leq & \int_\R e^{-w^2} \left(\frac{1}{\rho(x)^2}+w^2+\frac{w}{\rho(x)}\right) dw\\
            \leq & \frac{C}{\rho(x)^2} \int_\R e^{-w^2} (1+w^2+w) dw
            \leq \frac{C}{\rho(x)^2},
    \end{align*}
    because $\rho(x)\leq 1$. Hence,
    $$\int_\R e^{-c|x-z|^2/s}z^2 dz
        \leq C \frac{\sqrt{s}}{\rho(x)^2}, \quad x \in \R \text{ and } 0<s<1.$$
    It follows that
    \begin{align*}
        & \left| W_t^H(x,y)-W_t(x-y)\right|
            \leq C \Big( \int_0^{t/2} \int_\R \frac{e^{-c[|x-z|^2/s+|z-y|^2/(t-s)]}}{s^{1/2}(t-s)^{1/2}} z^2 dz ds \\
        & \qquad \qquad  + \int_0^{t/2} \int_\R \frac{e^{-c[|x-z|^2/(t-s)+|z-y|^2/s]}}{s^{1/2}(t-s)^{1/2}} z^2 dz ds \Big)\\
        & \qquad \leq  C \int_0^{t/2} \int_\R \frac{e^{-c[|x-z|^2+|z-y|^2]/t}}{t^{1/2}}
                                              \left(\frac{e^{-c|x-z|^2/s}}{s^{1/2}} + \frac{e^{-c|y-z|^2/s}}{s^{1/2}} \right)  z^2 dz ds \\
        & \qquad \leq  C \frac{e^{-c|x-y|^2/t}}{t^{1/2}} \int_0^{t/2} \int_\R
                                              \left(\frac{e^{-c|x-z|^2/s}}{s^{1/2}} + \frac{e^{-c|y-z|^2/s}}{s^{1/2}} \right) z^2 dz ds \\
        & \qquad \leq  C \frac{t}{\rho(x)^2} W_{ct}(x-y), \quad |x-y|<\rho(x) \text{ and } 0<t<\rho(x)^2,
    \end{align*}
    because $\rho(x)/C \leq \rho(y) \leq C \rho(x)$, provided that $|x-y|<\rho(x)$.
    Thus \eqref{3.3} is proved.

    By using \eqref{3.3} we deduce that
    \begin{align*}
        \left| \mathcal{D}^{loc}_{1}(h)(t,x) \right|
            \leq & C \int_0^t \int_{|x-y| < \rho(x)} |h(s,y)| \int_0^{\rho(x)^2}
                            \frac{e^{-c|t-s|^2/u}}{u^{1/2}\rho(x)^2} \frac{e^{-c|x-y|^2/u}}{u^{1/2}} du dy ds \\
            \leq & C \int_{|x-y| < \rho(x)} \left(\int_0^{\rho(x)^2} \frac{du}{u^{1/2}\rho(x)^2} \right) W_*(|h_0(\cdot,y)|)(t) dy \\
            \leq & \frac{C}{\rho(x)}  \int_{|x-y| < \rho(x)}W_*(|h_0(\cdot,y)|)(t) dy \\
            \leq & C \M\left[ W_*(|h_0|)(t) \right](x), \quad x \in \R \text{ and } t >0.
    \end{align*}
    Hence, $\mathcal{D}^{loc}_{1}$ is bounded from $L^2((0,T)\times \R)$ into itself. We conclude that $\mathcal{D}^{loc}$ is bounded
    in $L^2((0,T)\times \R,X)$.
\end{proof}

Note that the operators $K^{glob}(\{P_u^H\}_{u>0})$ and
$\mathcal{D}^{loc}$ are bounded in $L^2((0,T)\times \R,X)$ for
every Banach space $X$.

By Lemmas~\ref{Lem3.1} and \ref{Lem3.2} the operator
$K(\{P_u^H\}_{u>0})$ is bounded in $L^2((0,T)\times \R,X)$ if, and
only if, the operator $K^{loc}(\{P_u\}_{u>0})$ is bounded in
$L^2((0,T)\times \R,X)$. According to Theorem~\ref{ThA} our
equivalence ($(a) \Leftrightarrow (b)$ in Theorem~\ref{Th1} for
the Hermite operator $H$) will be showed once we establish that
$K(\{P_u\}_{u>0})$ is bounded in $L^2((0,T)\times \R,X)$ if, and
only if, $K^{loc}(\{P_u\}_{u>0})$ is bounded in $L^2((0,T)\times
\R,X)$. In order to do this we use some ideas developed in
\cite{AST} and \cite{HTV}.

\begin{Lem}\label{Lem3.3}
    Let $X$ be a Banach space and $T>0$. Then, the operator $K(\{P_u\}_{u>0})$ is bounded in $L^2((0,T)\times \R, X)$ if, and only if,
    $K^{loc}(\{P_u\}_{u>0})$ is bounded in $L^2((0,T)\times \R,X)$.
\end{Lem}

\begin{proof}
    Suppose that $K(\{P_u\}_{u>0})$ is bounded from $L^2((0,T)\times \R,X)$ into itself.
    Let $f \in C^\infty_c(0,T) \otimes C_c^\infty(\R)\otimes X$. There exists $C_0>0$  such that
    $$\frac{1}{C_0} \rho(x)
        \leq \rho(y)
        \leq C_0 \rho(x), \quad |x-y| \leq \rho(x).$$
    According to \cite[Proposition 5]{DGMTZ} (see also \cite[Lemma 2.3]{DZ1}) we choose a sequence $(x_k)_{k=1}^\infty \subseteq \R$ such that
    \begin{itemize}
        \item[$(i)$] $\R = \overset{\infty}{\underset{k=1}{\bigcup}} B(x_k,\rho(x_k))$.
        \item[$(ii)$] For every $M>0$ there exists $m \in \mathbb{N}$ such that, for every $\ell \in \mathbb{N}$,
        $$\#\{k \in \mathbb{N} : B(x_k,M\rho(x_k)) \cap B(x_\ell,M\rho(x_\ell)) \neq \varnothing\} \leq m,$$
        where $\#A$ denotes the cardinal of the set $A \subseteq \mathbb{N}$.
    \end{itemize}

    To simplify notation we write $I_k=B(x_k,\rho(x_k))$ and $MI_k=B(x_k,M\rho(x_k))$, $M>0$, $k \in \mathbb{N}$.

    We define the operator $T$ as follows
    $$T(f)(t,x)
        = \sum_{k=1}^\infty \chi_{I_k}(x) \int_0^t \int_{(C_0+1)I_k} \partial_t P_{t-s}(x-y)f(s,y)dyds, \quad x \in \R \text{ and } t>0.$$
    We have that
    \begin{align*}
        \left\| T(f) \right\|^2_{L^2((0,T)\times \R,X)}
            \leq & C \sum_{k=1}^\infty \left\| K(\{P_u\}_{u>0})\left(\chi_{(C_0+1)I_k}(y)f(s,y)\right) \right\|^2_{L^2((0,T)\times \R,X)}\\
            \leq & C \sum_{k=1}^\infty \left\|\chi_{(C_0+1)I_k}(y)f(s,y) \right\|^2_{L^2((0,T)\times \R,X)}\\
            \leq & C \left\| f \right\|^2_{L^2((0,T)\times \R,X)}.
    \end{align*}
    To see that $K^{loc}(\{P_u\}_{u>0})$ is bounded in $L^2((0,T)\times \R,X)$ it is enough to show that the difference operator
    $\D=T-K^{loc}(\{P_u\}_{u>0})$ is bounded in $L^2((0,T)\times \R,X)$.

    If $x \in I_k$ and $|x-y|<\rho(x)$, then $y \in (C_0+1)I_k$. Also, if $x \in I_k$ and $y \in (C_0+1)I_k$, then
    $$|x-y|
        \leq \rho(x_k) + (C_0+1) \rho(x_k)
        \leq C_0(C_0+2)\rho(x). $$
    Hence, we have that
    $$\left\{ (x,y) \in \R^2 : |x-y|<\rho(x) \right\}
        \subseteq \bigcup_{k \in \mathbb{N}} \left\{ (x,y) \in \R^2 : x \in I_k, \ y \in (C_0+1)I_k \right\},$$
    and
    \begin{align*}
        & \bigcup_{k \in \mathbb{N}} \left\{ (x,y) \in \R^2 : x \in I_k, \ y \in (C_0+1)I_k \right\}
                \setminus \left\{ (x,y) \in \R^2 : |x-y|<\rho(x) \right\} \\
        & \qquad \qquad \subseteq \left\{ (x,y) \in \R^2 : \rho(x) \leq |x-y| \leq  C_0(C_0+2) \rho(x) \right\}.
    \end{align*}

    We can write
    \begin{align*}
        \D(f)(t,x)
            = & \int_0^t \int_\R  \left[ \sum_{k=1}^\infty \chi_{I_k}(x) \chi_{(C_0+1)I_k}(y) - \chi_{\left\{ (x,y) \in \R^2 : |x-y|<\rho(x) \right\}}(x,y) \right] \\
            & \times \partial_tP_{t-s}(x-y)f(s,y) dy ds, \quad x \in \R \text{ and } t>0.
    \end{align*}
    Hence, we get
    \begin{align*}
        \| \D(f)(t,x) \|_X
            \leq & C \int_0^t \int_{\rho(x) \leq |x-y| \leq  C_0(C_0+2) \rho(x)} \left|\partial_tP_{t-s}(x-y) \right| \|f(s,y)\|_X dy ds \\
            \leq & C \int_0^t \int_{\rho(x) \leq |x-y| \leq  C_0(C_0+2) \rho(x)} \frac{1}{|t-s|^2+|x-y|^2}\|f(s,y)\|_X dy ds \\
            \leq & C \int_0^t \int_{\rho(x) \leq |x-y| \leq  C_0(C_0+2) \rho(x)} \frac{1}{|t-s|^2+\rho(x)^2}\|f(s,y)\|_X dy ds \\
            \leq & \frac{C}{\rho(x)} \int_{|x-y| \leq  C_0(C_0+2) \rho(x)} P_*\left(\|f_0(\cdot,y)\|_X\right)(t) dy  \\
            \leq & C \M\left[P_*\left(\|f_0\|_X\right)(t)\right](x), \quad x \in \R \text{ and } t>0.
    \end{align*}
    We conclude that $\D$ is bounded from $L^2((0,T)\times \R,X)$ into itself. Then, $K^{loc}(\{P_u\}_{u>0})$
    is bounded in $L^2((0,T)\times \R,X)$.

    Assume now that $K^{loc}(\{P_u\}_{u>0})$ is a bounded operator from $L^2((0,T)\times \R,X)$ into itself.
    We are going to see that $K(\{P_u\}_{u>0})$ is bounded in $L^2((0,T)\times \R,X)$.
    Let $f \in C_c^\infty(0,T) \otimes C_c^\infty(\R) \otimes X$.

    For every $R>0$ we define $f_R(t,x)=f(t,Rx)$ and
    $f^R(t,x)=f(Rt,x)$, $x \in \R$ and $t>0$. Straightforward manipulations allows us to show that
    $$K(\{P_u\}_{u>0})(f)(t,x)
        = K(\{P_u\}_{u>0})((f_R)^R)\left(\frac{t}{R}, \frac{x}{R}\right), \quad x \in \R \text{ and } t,R >0.$$
    Suppose that $f(t,x)=0$, $|x|>a$, $t>0$, where $a>0$. Fix $b>0$. We have that, for every $R>0$,
    $$\rho\left(\frac{x}{R}\right)
        = \left\{\begin{array}{ll}
            1/2, & |x| \leq R, \\
            \dfrac{R}{R+|x|} > \dfrac{R}{R+b}, & |x|>R,
          \end{array}\right.
        \quad |x| \leq b,$$
    and
    $$\left| y - \frac{x}{R} \right|
        \leq \frac{a+b}{R}, \quad |y| \leq \frac{a}{R} \text{ and } |x| \leq b. $$
    Hence, there exists $R_0>0$ such that $|y-x/R|<\rho(x/R)$, $|x|\leq b$, $|y| \leq a/R$ and $R \geq R_0$.
    Hence,
    \begin{align*}
        K(\{P_u\}_{u>0})(f)(t,x)
            = & \frac{1}{\pi} \int_0^{t/R} \int_{-a/R}^{a/R}
                \frac{\left( t/R - s \right)^2 - \left( x/R - y \right)^2}{\left[ \left( t/R - s \right)^2 + \left( x/R - y \right)^2 \right]^2} (f_R)^R(s,y)dy ds \\
            &   = \frac{1}{\pi} \int_0^{t/R} \int_{|y-x/R| < \rho(x/R)}
                \frac{\left( t/R - s \right)^2 - \left( x/R - y \right)^2}{\left[ \left( t/R - s \right)^2 + \left( x/R - y \right)^2 \right]^2} (f_R)^R(s,y)dy ds \\
            = & K^{loc}(\{P_u\}_{u>0})((f_R)^R)\left(\frac{t}{R}, \frac{x}{R}\right), \quad |x|\leq b, \ t >0 \text{ and } R \geq R_0.
    \end{align*}
    Since $K^{loc}(\{P_u\}_{u>0})$ is bounded in $L^2((0,T)\times \R,X)$, it follows that
    \begin{align*}
        \int_0^T \int_{-b}^b \left\| K(\{P_u\}_{u>0})(f)(t,x) \right\|_X^2 dx dt
            = & \int_0^T \int_\R \left\| K^{loc}(\{P_u\}_{u>0})((f_{R_0})^{R_0})\left(\frac{t}{{R_0}}, \frac{x}{{R_0}}\right) \right\|_X^2 dx dt \\
            = & R_0^2 \int_0^T \int_\R \left\| K^{loc}(\{P_u\}_{u>0})((f_{R_0})^{R_0})\left(t, x \right) \right\|_X^2 dx dt \\
            \leq & C R_0^2 \int_0^T \int_\R \left\| f \left(R_0t, R_0x \right) \right\|_X^2 dx dt \\
            \leq & C \|f\|^2_{L^2((0,T)\times \R,X)}.
    \end{align*}
    Hence,
    $$\| K(\{P_u\}_{u>0})(f) \|^2_{L^2((0,T)\times \R,X)}
        \leq  C \|f\|^2_{L^2((0,T)\times \R,X)}.$$
\end{proof}

\subsection{Ornstein-Uhlenbeck operator $\mathcal{O}+1/2$}

We use again Lemma~\ref{LemTransfer}, with $\overline{\mathcal{L}}= \mathcal{O}+1/2$, $\mathcal{L}=H$, $M(x)= \pi^{-1/4}e^{-x^2/2}$
and $h(x)=x$ to get
$$\mathcal{O}+ \frac{1}{2} = (U \circ W)^{-1} \circ H \circ (U \circ W).$$

%\newpage
%%%%%%%%%%%%%%%%%%%%%%%%%%%%%%%%%%%%%%%%%%%%%%%%%%%%%%%%%%%%%%%%%%%%%%%%%%%%%%%%%%%%%%%%%%%%%%%%%%%%%%%%%%%%%%%%%%%%
\section{Proof of Theorem~\ref{Th1} for Laguerre type operators}\label{sec:Laguerre}
%%%%%%%%%%%%%%%%%%%%%%%%%%%%%%%%%%%%%%%%%%%%%%%%%%%%%%%%%%%%%%%%%%%%%%%%%%%%%%%%%%%%%%%%%%%%%%%%%%%%%%%%%%%%%%%%%%%%

\subsection{Laguerre operator $L_\alpha^\varphi$}

We consider the Laguerre differential operator $L_\alpha^\varphi$, $\alpha>0$, defined by
$$L_\alpha^\varphi={1\over 2}\left(-{{d^2}\over{dx^2}}+x^2+{{\alpha(\alpha-1)}\over{x^2}}\right), \quad x\in (0,\infty).$$

We have that
$$L_\alpha^\varphi\varphi_k^{\alpha}=\left(2k+\alpha +{1\over 2}\right)\varphi_k^{\alpha},$$
where
$$ \varphi_k^{\alpha}(x)
    =\left({{2\Gamma(k+1)}\over{\Gamma(k+\alpha + {1/2})}}\right)^{1/ 2}e^{-{{x^2}/ 2}}x^{\alpha }\ell_k^{\alpha-{1/ 2}}(x^2), \quad x \in (0,\infty),$$
and, for every $k\in\N$, $\ell_k^{\alpha}$ represents the $k$-th Laguerre polynomial of order $\alpha$ (\cite[pp. 100--102]{Sz}).

The operator $-L_\alpha^\varphi$ generates the positive bounded semigroup $\{W_t^{L_\alpha^\varphi}\}_{t>0}$, that is contractive when $\alpha >1$
(\cite[Theorem 4.1]{NoSt3}), in $L^p(0,\infty)$,
where, for every $f\in L^p (0,\infty)$, $1\leq p\leq\infty$,
$$W_t^{L_\alpha^\varphi}(f)(x)=\int_0^{\infty}{W_t^{L_\alpha^\varphi}(x,y)f(y)dy}, \quad x\in (0,\infty), \ t>0,$$
being, for every $t,x,y\in (0,\infty)$,
$$W_t^{L_\alpha^\varphi}(x,y)
    =  \frac{2 \sqrt{xy}e^{-t}}{1-e^{-2t}}
        I_{\alpha-{1/ 2}}\left({{2xye^{-t}}\over{1-e^{-2t}}}\right)
        \exp\left[-{1\over 2}{{1+e^{-2t}}\over{1-e^{-2t}}}(x^2+y^2)\right].$$

By using the subordination formula we can write the Poisson kernel $P_t^{L_\alpha^\varphi}(x,y)$, $t,x,y\in (0,\infty)$, associated with $L_\alpha^\varphi$ as follows
$$P_t^{L_\alpha^\varphi}(x,y)
    ={t\over{\sqrt{4\pi}}}\int_0^{\infty}{{{e^{-{{t^2}/{4u}}}}\over{u^{{3/ 2}}}}W_u^{L_\alpha^\varphi}(x,y)du}, \quad t,x,y\in(0,\infty).$$

The Poisson semigroup for the Laguerre operator $L_\alpha^\varphi$ is defined by
$$P_t^{L_\alpha^\varphi}(f)(x)=\int_0^{\infty}{P_t^{L_\alpha^\varphi}(x,y)f(y)dy}, \quad t,x\in(0,\infty),$$
for every $f\in L^p(0,\infty)$, $1\leq p\leq\infty$.

The starting point for the harmonic analysis in Laguerre context are the works of Muckenhoupt \cite{Mu1} and \cite{Mu3}.
In the last years we remark the results of Nowak and Stempak \cite{NoSt1} and \cite{NoSt2} in higher dimensions.
Laguerre harmonic analysis in a Banach valued setting has been developed in \cite{BFRST1} and \cite{BFRST2}.

We consider the operator
$$K\left(\{P_u^{L_\alpha^\varphi}\}_{u>0}\right)(f)(t,x)
    = \int_0^t \partial_t P_{t-s}^{L_\alpha^\varphi}(f(s,\cdot))(x)ds, \quad t,x \in (0,\infty), $$
for every $f \in C_c^\infty(0,T) \otimes C_c^\infty(0,\infty) \otimes X$.

Since we showed in Section~\ref{sec:Hermite}, $(a) \Leftrightarrow (b)$ in Theorem~\ref{Th1} for the Hermite operator $H$,
it is enough to prove that $K(\{P_u^{L_\alpha^\varphi}\}_{u>0})$ is bounded in $L^2((0,T) \times (0,\infty), X)$ if, and only if,
$K\left(\{P_u^{H}\}_{u>0}\right)$ is bounded in $L^2((0,T) \times \R, X)$.

We consider the operator
$$K_+\left(\{P_u^{H}\}_{u>0}\right)(f)(t,x)
    = \int_0^t \int_0^\infty \partial_t P_{t-s}^{H}(x,y) f(s,y) dy ds, \quad t,x \in (0,\infty), $$
for every $f \in C_c^\infty(0,T) \otimes C_c^\infty(0,\infty) \otimes X$.

As in Lemma~\ref{Lem2.1}, we reduce the boundedness of $K\left(\{P_u^{H}\}_{u>0}\right)$ in $L^2((0,T) \times \R, X)$
to study the operator $K_+\left(\{P_u^{H}\}_{u>0}\right)$ in $L^2((0,T) \times (0,\infty), X)$.

\begin{Lem}\label{Lem4.1}
    Let $X$ be a Banach space and $T>0$. Then, the operator $K\left(\{P_u^{H}\}_{u>0}\right)$ is bounded in $L^2((0,T)\times \R,X)$ if, and only if,
    $K_+\left(\{P_u^{H}\}_{u>0}\right)$ is bounded in $L^2((0,T) \times (0,\infty),X)$.
\end{Lem}

\begin{proof}
    Assume firstly that $K\left(\{P_u^{H}\}_{u>0}\right)$ is a bounded operator from $L^2((0,T)\times \R,X)$ into itself, and
    let $f \in C_c^\infty(0,T) \otimes C_c^\infty(0,\infty) \otimes X$. We define the extension $f_0$ of $f$ as in \eqref{5.1}.
    Then,
    $$K_+\left(\{P_u^{H}\}_{u>0}\right)(f)(t,x)
        = K\left(\{P_u^{H}\}_{u>0}\right)(f_0)(t,x), \quad t,x \in (0,\infty),$$
    and
    \begin{align*}
        \| K_+\left(\{P_u^{H}\}_{u>0}\right)(f) \|_{L^2((0,T) \times (0,\infty),X)}
            \leq & \| K\left(\{P_u^{H}\}_{u>0}\right)(f_0) \|_{L^2((0,T) \times \R, X)} \\
            \leq & C \| f \|_{L^2((0,T) \times (0,\infty),X)}.
    \end{align*}

    Suppose now that $K_+\left(\{P_u^{H}\}_{u>0}\right)$ is bounded from $ L^2((0,T) \times (0,\infty), X)$ into itself.
    Let $f \in C_c^\infty(0,T) \otimes C_c^\infty(0,\infty) \otimes X$. We define
    operator $K_-\left(\{P_u^{H}\}_{u>0}\right)$ through
    $$K_-\left(\{P_u^{H}\}_{u>0}\right)(f)(t,x)
        = \int_0^t \int_0^\infty \partial_t P_{t-s}^{H}(x,-y) f(s,y) dy ds, \quad t,x \in (0,\infty).$$
    We can write
    \begin{align*}
        & \int_0^T \int_\R \left\| K\left(\{P_u^{H}\}_{u>0}\right)(f)(t,x) \right\|_X^2 dx dt
            \leq 2 \Big( \int_0^T \int_0^\infty \left\| K_+\left(\{P_u^{H}\}_{u>0}\right)(f_+)(t,x) \right\|_X^2 dx dt \\
        & \qquad + \int_0^T \int_0^\infty \left\| K_-\left(\{P_u^{H}\}_{u>0}\right)(f_-)(t,x) \right\|_X^2 dx dt
                        + \int_0^T \int_0^\infty \left\| K_+\left(\{P_u^{H}\}_{u>0}\right)(f_-)(t,x) \right\|_X^2 dx dt \\
        & \qquad + \int_0^T \int_0^\infty \left\| K_-\left(\{P_u^{H}\}_{u>0}\right)(f_+)(t,x) \right\|_X^2 dx dt\Big),
    \end{align*}
    where $f_+$ and $f_-$ were defined in the proof of Lemma~\ref{Lem2.1}.
    In order to see that the operator $K\left(\{P_u^{H}\}_{u>0}\right)$ is bounded from $L^2((0,T) \times \R, X)$
    into itself, it is enough to prove that $K_-\left(\{P_u^{H}\}_{u>0}\right)$ is bounded in $L^2((0,T) \times (0,\infty), X)$.

    We have that
    $$ \left| \partial_t P_{t-s}^H(x,-y) \right|
        \leq C \int_0^\infty \frac{e^{-c[|t-s|^2+ (x+y)^2]/u}}{u^2} du
        \leq \frac{C}{|t-s|^2 + (x+y)^2}, \quad s,t,x,y \in (0,\infty). $$
    Then, by proceeding as in the proof of Lemma~\ref{Lem2.1} we conclude that $K_-\left(\{P_u^{H}\}_{u>0}\right)$
    is bounded from $L^2((0,T) \times (0,\infty), X)$ into itself.
\end{proof}

We now split the operator $K_+\left(\{P_u^{H}\}_{u>0}\right)$ as follows
$$K_+\left(\{P_u^{H}\}_{u>0}\right)
    = K_+^{loc}\left(\{P_u^{H}\}_{u>0}\right) + K_+^{glob}\left(\{P_u^{H}\}_{u>0}\right),$$
where
$$ K_+^{loc}\left(\{P_u^{H}\}_{u>0}\right)(f)(t,x)
    = \int_0^t \int_{x/2}^{2x} \partial_t P_{t-s}^H(x,y)f(s,y) dy ds, \quad t,x \in (0,\infty),$$
for every $f \in C_c^\infty(0,T) \otimes C_c^\infty(0,\infty) \otimes X$.

The operator $K(\{P_u^{L_\alpha^\varphi}\}_{u>0})$ is also split in the following way
$$K\left(\{P_u^{L_\alpha^\varphi}\}_{u>0}\right)
    = K^{loc}\left(\{P_u^{L_\alpha^\varphi}\}_{u>0}\right) + K^{glob}\left(\{P_u^{L_\alpha^\varphi}\}_{u>0}\right),$$
where
$$ K^{loc}\left(\{P_u^{L_\alpha^\varphi}\}_{u>0}\right)(f)(t,x)
    = \int_0^t \int_{x/2}^{2x} \partial_t P_{t-s}^{L_\alpha^\varphi}(x,y)f(s,y) dy ds, \quad t,x \in (0,\infty),$$
for every $f \in C_c^\infty(0,T) \otimes C_c^\infty(0,\infty) \otimes X$.

In the sequel we will use the following properties of the modified Bessel function $I_\nu$, that can be
encountered in \cite[pp. 136 and 123]{Leb}. The behavior of $I_\nu(z)$ close to zero is
\begin{equation}\label{4.1}
    I_\nu(z)
        \sim \frac{1}{2^\nu \Gamma(\nu+1)} z^\nu, \quad \text{as } z \to 0^+.
\end{equation}
Moreover, for every $n \in \mathbb{N}$,
\begin{equation}\label{4.2}
    \sqrt{z} I_\nu(z)
        = \frac{e^z}{\sqrt{2\pi}} \left[ \sum_{k=0}^n \frac{(-1)^k [\nu,k]}{(2z)^{k}} + \mathcal{O}\left( \frac{1}{z^{n+1}} \right) \right],
\end{equation}
where $[\nu,0]=1$ and
$$[\nu,k]
    = \frac{(4\nu^2-1)(4\nu^2-3^2)\cdots (4\nu^2-(2k-1)^2)}{2^{2k}\Gamma(k+1)}, \quad k=1,2, \dots.$$

\begin{Lem}\label{Lem4.2}
    Let $X$ be a Banach space and $T>0$. Then, the operators $K_+^{glob}(\{P_u^H\}_{u>0})$
    and $K^{glob}\left(\{P_u^{L_\alpha^\varphi}\}_{u>0}\right)$ are bounded in $L^2((0,T)\times (0,\infty),X)$.
\end{Lem}

\begin{proof}
    Let $f \in C_c^\infty(0,T) \otimes C_c^\infty(0,\infty) \otimes X$.
    We have that
    $$ \left| \partial_t P_{t-s}^H(x,y) \right|
        \leq C \int_0^\infty \frac{e^{-c[|t-s|^2+ |x-y|^2]/u}}{u^2} du
        \leq \frac{C}{|t-s|^2 + |x-y|^2}, \quad s,t,x,y \in (0,\infty). $$
    Hence, we obtain
    \begin{align*}
        & \| K_+^{glob}(\{P_u^H\}_{u>0})(f)(t,x) \|_X
            \leq  C \Big( \int_0^t \int_0^{x/2} \frac{1}{|t-s|^2 + x^2} \|f(s,y)\|_X dy ds \\
        & \qquad \qquad +  \int_0^t \int_{2x}^\infty \frac{1}{|t-s|^2 + y^2} \|f(s,y)\|_X dy ds \Big) \\
        & \qquad    \leq  C \left( \frac{1}{x} \int_0^x P_*\left( \|\tilde{f}_0(\cdot,y)\|_X \right)(t)dy
                        +  \int_y^\infty \frac{1}{y} P_*\left( \|\tilde{f}_0(\cdot,y)\|_X \right)(t)dy \right), \quad  t,x \in (0,\infty).
    \end{align*}
    Since the Hardy operators $H_0$ and $H_\infty$ (see Section~\ref{sec:Bessel}) are bounded in $L^2(0,\infty)$ and the maximal operator $P_*$
    is bounded in $L^2(\R)$, we deduce that $K_+^{glob}(\{P_u^H\}_{u>0})$ is bounded in $L^2((0,T) \times (0,\infty), X)$.

    According to \eqref{4.1} and \eqref{4.2} we get
    \begin{equation*}\label{18.1}
        0 \leq W_t^{L_\alpha^\varphi}(x,y)
            \leq C \frac{e^{-c|x-y|^2/t}}{\sqrt{t}}, \quad t,x,y \in (0,\infty).
    \end{equation*}
    By proceeding as above we obtain, for every $t,x \in (0,\infty)$,
    \begin{align*}
        & \| K^{glob}(\{P_u^{L_\alpha^\varphi}\}_{u>0})(f)(t,x) \|_X
        \leq  C \left( H_0\left[P_*\left( \|\tilde{f}_0(\cdot,y)\|_X \right)(t)\right](x)
    +                  H_\infty\left[P_*\left( \|\tilde{f}_0(\cdot,y)\|_X \right)(t)\right](x) \right).
    \end{align*}
    Hence,  $K^{glob}(\{P_u^{L_\alpha^\varphi}\}_{u>0})$ is a bounded operator from $L^2((0,T) \times (0,\infty), X)$ into itself.
\end{proof}

From Lemmas~\ref{Lem4.1} and \ref{Lem4.2} we deduce that the equivalence $(a) \Leftrightarrow (b)$ in
Theorem~\ref{Th1} for $L_\alpha^\varphi$ will be established once we prove the following.

\begin{Lem}\label{Lem4.3}
    Let $X$ be a Banach space and $T>0$. Then, the operator $\DD^{loc}$ given by
    $$\DD^{loc}=K_+^{loc}\left(\{P_u^{H}\}_{u>0}\right) - K^{loc}\left(\{P_u^{L_\alpha^\varphi}\}_{u>0}\right),$$
    is bounded in $L^2((0,T) \times (0,\infty),X)$.
\end{Lem}

\begin{proof}
    Let $f \in C_c^\infty(0,T) \otimes C_c^\infty(0,\infty) \otimes X$. We can write
    \begin{align*}
        \DD^{loc}(f)(t,x)
            = & \frac{1}{\sqrt{4\pi}} \int_0^t \int_0^\infty \frac{e^{-|t-s|^2/4u}}{u^{3/2}} \left(1 - \frac{|t-s|^2}{2u} \right) \\
            & \times \int_{x/2}^{2x} \left[ W_u^{H}(x,y) - W_u^{L_\alpha^\varphi}(x,y) \right] f(s,y) dyduds, \quad t,x \in (0,\infty).
    \end{align*}
    By defining the function $\tilde{f}$ as in \eqref{ext}, we get for every $t,x \in (0,\infty)$,
    \begin{equation}\label{4.*}
        \|\DD^{loc}(f)(t,x)\|_X
            \leq C \int_{x/2}^{2x} W_*(\|\tilde{f}(\cdot,y)\|_X)(t) \left( \int_0^\infty \left| W_u^{H}(x,y) - W_u^{L_\alpha^\varphi}(x,y) \right| \frac{du}{u} \right) dy.
    \end{equation}
    In order to estimate the function
    $$H(x,y)
        = \int_0^\infty \left| W_u^{H}(x,y) - W_u^{L_\alpha^\varphi}(x,y) \right| \frac{du}{u}, \quad x,y \in (0,\infty), \ x \neq y,$$
    we distingue two cases.
    To simplify the notation we call $\xi=\xi(u,x,y)=2xye^{-u}/(1-e^{-2u})$, $u,x,y \in (0,\infty)$.

    Suppose firstly that $\xi \geq 1$. By \eqref{10.1} and \eqref{4.2} we deduce
    \begin{align*}
        \left| W_u^{H}(x,y) - W_u^{L_\alpha^\varphi}(x,y) \right|
            = & W_u^{H}(x,y) \left| 1 - \sqrt{2\pi}e^{-\xi} \sqrt{\xi} I_{\alpha-1/2}(\xi) \right|
            \leq \frac{C}{\xi} W_u^{H}(x,y) \\
            \leq & C \left\{
                        \begin{array}{ll}
                            \dfrac{e^{-c|x-y|^2/u}}{\xi^{1/4}\sqrt{u}}, & 0 < u < 1,\\
                            &\\
                            \dfrac{e^{-u/2}e^{-c|x-y|^2}}{\xi^{1/4}}, & u \geq  1.
                        \end{array} \right.
    \end{align*}
    Then,
    \begin{align}\label{19.1}
        \int_{0, \ \xi \geq 1}^1 \left| W_u^{H}(x,y) - W_u^{L_\alpha^\varphi}(x,y) \right| \frac{du}{u}
            \leq & C \int_0^1 \dfrac{e^{-c|x-y|^2/u}}{\xi^{1/4}} \frac{du}{u^{3/2}}
            \leq \frac{C}{\sqrt{y}}\int_0^1 \dfrac{e^{-c|x-y|^2/u}}{u^{5/4}} du \nonumber \\
            \leq & \frac{C}{\sqrt{y|x-y|}}\int_{0}^\infty \dfrac{e^{-v}}{v^{3/4}} dv
            \leq  \frac{C}{y} \sqrt{\frac{y}{|x-y|}}, \quad \frac{x}{2}<y<2x, \ x \neq y.
    \end{align}
    Also, we obtain
    \begin{align}\label{19.2}
        \int_{1, \ \xi \geq 1}^\infty \left| W_u^{H}(x,y) - W_u^{L_\alpha^\varphi}(x,y) \right| \frac{du}{u}
            \leq & C \int_1^\infty \dfrac{e^{-u/2}e^{-c|x-y|^2}}{\xi^{1/4}} \frac{du}{u}
            \leq \frac{e^{-c|x-y|^2}}{\sqrt{y}}\int_1^\infty e^{-u/4} du \nonumber \\
            \leq &  \frac{C}{y} \sqrt{\frac{y}{|x-y|}}, \quad \frac{x}{2}<y<2x, \ x \neq y.
    \end{align}

    On the other hand, if $\xi \leq 1$, by \eqref{4.1} we have that
    \begin{align*}
        \left| W_u^{H}(x,y) - W_u^{L_\alpha^\varphi}(x,y) \right|
            \leq & C W_u^{H}(x,y) \left( 1+ e^{-\xi} \sqrt{\xi} I_{\alpha-1/2}(\xi)  \right)
            \leq C W_u^{H}(x,y) \left( 1 + \xi^\alpha \right) \\
            \leq & C \left\{
                        \begin{array}{ll}
                            \left( \dfrac{e^{-u}}{1-e^{-2u}} \right)^{1/4}\dfrac{e^{-c|x-y|^2/u}}{u^{1/4}}, & 0 < u < 1,\\
                            &\\
                            \left( \dfrac{e^{-u}}{1-e^{-2u}} \right)^{1/4} e^{-cu}e^{-c|x-y|^2}, & u \geq  1.
                        \end{array} \right.
    \end{align*}
    Then,
    \begin{align}\label{19.3}
        \int_{0, \ \xi \leq 1}^1 \left| W_u^{H}(x,y) - W_u^{L_\alpha^\varphi}(x,y) \right| \frac{du}{u}
            \leq & C \int_{0 , \ \xi \leq 1}^1 \dfrac{e^{-c|x-y|^2/u}}{u^{5/4}} \frac{\xi^{1/4}}{\sqrt{y}} du \nonumber \\
            \leq & \frac{C}{\sqrt{y}} \int_0^1 \dfrac{e^{-c|x-y|^2/u}}{u^{5/4}} du
            \leq  \frac{C}{y} \sqrt{\frac{y}{|x-y|}}, \quad \frac{x}{2}<y<2x, \ x \neq y,
    \end{align}
    and
    \begin{align}\label{20.1}
        \int_{1, \ \xi \leq 1}^\infty \left| W_u^{H}(x,y) - W_u^{L_\alpha^\varphi}(x,y) \right| \frac{du}{u}
            \leq & C \int_{1, \ \xi \leq 1}^\infty \frac{\xi^{1/4}}{\sqrt{y}} e^{-cu}  e^{-c|x-y|^2} du \nonumber \\
            \leq & C \frac{e^{-c|x-y|^2}}{\sqrt{y}} \int_1^\infty e^{-u/4} du
            \leq  \frac{C}{y} \sqrt{\frac{y}{|x-y|}}, \quad \frac{x}{2}<y<2x, \ x \neq y.
    \end{align}

    From \eqref{4.*}, \eqref{19.1}, \eqref{19.2}, \eqref{19.3} and \eqref{20.1} we deduce that
    \begin{equation}\label{4.6}
        \|\DD^{loc}(f)(t,x)\|_X
            \leq C \mathcal{N}\left[ W_*\left( \|\tilde{f}\|_X \right)(t) \right](x), \quad t,x \in (0,\infty),
    \end{equation}
    being
    $$\mathcal{N}(f)
        = \int_{x/2}^{2x} \frac{1}{y} \sqrt{\frac{y}{|x-y|}} f(y) dy, \quad f \in L^2(0,\infty).$$
    Jensen's inequality implies that the operator $\mathcal{N}$ is bounded from $L^2(0,\infty)$ into itself. Also,
    the maximal operator $W_*$ is bounded in $L^2(\R)$. Then, from \eqref{4.6} we infer that $\DD^{loc}$ is a bounded operator
    from $L^2((0,T) \times (0,\infty), X)$ into itself.
\end{proof}

\subsection{Laguerre operators $L_\alpha + (\alpha + 1/2)/2$, $L_\alpha^\ell$, $L_\alpha^\psi$ and $L_\alpha^\mathpzc{L}$}

Once again we use Lemma~\ref{LemTransfer} with the appropriate isometry and change of variables,
\begin{itemize}
    \item If $\overline{\mathcal{L}}=L_\alpha + (\alpha + 1/2)/2$, then $\mathcal{L}=L_\alpha^\varphi$, $M(x)=\sqrt{2} x^{\alpha}e^{-x^2/2}$, $h(x)=x^2$ and
    $$\Big(L_\alpha + \frac{\alpha + 1/2}{2}\Big) = (U \circ W)^{-1} \circ L_\alpha^\varphi \circ (U \circ W).$$
    \item If $\overline{\mathcal{L}}=L_\alpha^\ell$, then $\mathcal{L}=L_\alpha^\varphi$, $M(x)=\sqrt{2} x^{\alpha}$, $h(x)=x^2$ and
    $$L_\alpha^\ell  =  (U \circ W)^{-1} \circ L_\alpha^\varphi \circ (U \circ W).$$
    \item If $\overline{\mathcal{L}}=L_\alpha^\psi$, then $\mathcal{L}=L_\alpha^\varphi$, $M(x)=x^{\alpha}$, $h(x)=x$ and
    $$L_\alpha^\psi =  (U \circ W)^{-1} \circ L_\alpha^\varphi \circ (U \circ W).$$
    \item If $\overline{\mathcal{L}}=L_\alpha^{\mathpzc{L}}$, then $\mathcal{L}=L_\alpha^\varphi$, $M(x)=\sqrt{2} x^{1/2}$, $h(x)=x^2$ and
    $$L_\alpha^{\mathpzc{L}} = (U \circ W)^{-1} \circ L_\alpha^\varphi \circ (U \circ W).$$
\end{itemize}

\quad \\
\textbf{Acknowledgements}. We would like to offer thanks to the referee for his advice which have improved quite a lot this paper.

%\newpage
%%%%%%%%%%%%%%%%%%%%%%%%%%%%%%%%%%%%%%%%%%%%%%%%%%%%%%%%%%%%%%%%%%%%%%%%%%%%%%%%%%%%%%%%%%%%%%%%%%%%%%%%%%%%%%%%%%%%
%       REFERENCIAS
%%%%%%%%%%%%%%%%%%%%%%%%%%%%%%%%%%%%%%%%%%%%%%%%%%%%%%%%%%%%%%%%%%%%%%%%%%%%%%%%%%%%%%%%%%%%%%%%%%%%%%%%%%%%%%%%%%%%

%\bibliographystyle{siam}
%\bibliography{references2}

\end{document}